\documentclass[lettersize,journal]{IEEEtran}
\usepackage{figlatex}  

\usepackage{enumitem}
\usepackage[utf8]{inputenc} 
\usepackage[T1]{fontenc}    
\usepackage{hyperref}       
\usepackage{url}            
\usepackage{amsfonts}       
\usepackage{nicefrac}       
\usepackage{microtype}      
\usepackage{xcolor}         

\usepackage{textcomp} 
\usepackage[english]{babel} 
\usepackage[autolanguage]{numprint} 
\usepackage{hyperref} 
\usepackage{verbatim} 
\usepackage{amsmath,amsthm} 
\usepackage{dsfont}
\usepackage[linesnumbered,ruled]{algorithm2e} 
\usepackage{xspace}
\usepackage{tikz}

\newtheorem{theorem}{Theorem}
\newtheorem{lemma}{Lemma}

\newtheorem{remark}{Remark}

\newtheorem{proposition}{Proposition}

\newtheorem{definition}{Definition}

\newtheorem{assumption}{Assumption}

\newcommand{\N}{\mathbb{N}}
\newcommand{\R}{\mathbb{R}}
\newcommand{\E}{{\mathbb{E}}}
\newcommand{\Pb}{{\mathbb{P}}}

\newcommand{\prt}[1]{\left(#1\right)}        
\newcommand{\brac}[1]{\left[#1\right]}       
\newcommand{\croc}[1]{\left\{#1\right\}}     
\newcommand{\norm}[1]{\left\|#1\right\|}     
\newcommand{\abs}[1]{\left|#1\right|}        

\newcommand{\eps}{\varepsilon}               
\newcommand{\I}[1]{\mathbb{I}_{\{#1\}}}

\newcommand{\btermn}{\Delta_{{\rm diff},n}}               
\newcommand{\btermi}{\Delta_{{\rm diff},i}}               
\newcommand{\ntermn}{\Delta_{{\rm mart},n}}               
\newcommand{\ntermi}{\Delta_{{\rm mart},i}}               
\newcommand{\mtermn}{\Delta_{{\rm mix},n}}
\newcommand{\mtermi}{\Delta_{{\rm mix},i}}
\newcommand{\tsimul}{\tau}   
\newcommand{\stepsize}{\gamma}   
\newcommand{\costOg}{G}   

\newcommand{\cost}{F}   
\newcommand{\pen}{b}   

\allowdisplaybreaks

\newif\ifSOC
\SOCfalse

\ifSOC
     \newcommand{\review}[1]{{\color{black}#1}}
\else
    \newcommand{\review}[1]{{\color{black}#1}}
\fi

\begin{document}

\title{Non-Stationary Gradient Descent for Optimal Auto-Scaling in Serverless Platforms}


\author{Jonatha~Anselmi, Bruno Gaujal, Louis-S\'ebastien Rebuffi
\thanks{J. Anselmi, B. Gaujal and L.-S. Rebuffi are with
Univ. Grenoble Alpes, CNRS, Inria, Grenoble INP, LIG, 38000 Grenoble, France.
Email: firstname.lastname@inria.fr
}
}

\markboth{Journal of \LaTeX\ Class Files,~Vol.~14, No.~8, August~2021}%
{Shell \MakeLowercase{\textit{et al.}}: A Sample Article Using IEEEtran.cls for IEEE Journals}


\maketitle


\begin{abstract}
To efficiently manage serverless computing platforms, a key aspect is the auto-scaling of services, i.e., the set of computational resources allocated to a service adapts over time as a function of the traffic demand. The objective is to find a compromise between user-perceived performance and energy consumption.
In this paper, we consider the \emph{scale-per-request} auto-scaling pattern and investigate how many function instances (or servers) should be spawned each time an \emph{unfortunate} job arrives, i.e., a job that finds all servers busy upon its arrival.
We address this problem by following a stochastic optimization approach:
we develop a stochastic gradient descent scheme of the Kiefer--Wolfowitz type that applies \emph{over a single run of the state evolution}.
At each iteration, the proposed scheme computes an estimate of the number of servers to spawn each time an unfortunate job arrives to minimize some cost function. Under natural assumptions, we show that the sequence of estimates produced by our scheme is asymptotically optimal almost surely. In addition, we prove that its convergence rate  is $O(n^{-2/3})$ where $n$ is the number of iterations.

From a mathematical point of view, the stochastic optimization framework induced by auto-scaling exhibits non-standard aspects that we approach from a general point of view.
We consider the  setting where a  controller can only get samples of the \emph{transient} -- rather than stationary -- behavior of the underlying stochastic system. To handle this difficulty, we develop arguments that exploit properties of the mixing time of the underlying Markov chain. By means of numerical simulations, we validate the proposed approach and quantify its gain with respect to  common existing scale-up rules.

\end{abstract}


\section{Introduction}

\subsection{Auto-scaling in Serverless Computing}

Auto-scaling mechanisms are considered to be essential components of serverless computing systems as they efficiently support cloud providers in handling the largest possible user base on their physical platforms.
These mechanisms are designed to automatically adjust the current service capacity in response to the current load while ensuring that service level agreement (SLA) contracts are respected.
In this paper, we tweak a popular auto-scaling paradigm that in the cloud computing literature is known as ``per-request''~\cite{scaleperrequest,ASLANPOUR2024266,Anselmi24} or ``reactive''~\cite{DOGANI2024104837} auto-scaling. According to this paradigm, an incoming request (or \emph{job}) is processed by an active idle function instance (or \emph{server}) if there is any available,
otherwise, the platform spawns a new server that will serve the job immediately after a \emph{coldstart} latency.
By design, the activation of a new server is only triggered at the arrival time of an \emph{unfortunate} job, i.e., a job that finds no active idle servers upon its arrival and thus must wait.
In practice, this is the de-facto auto-scaling pattern and is currently employed by serverless computing platforms such as AWS Lambda, Google Cloud Functions, Azure Functions, IBM Cloud Functions and Apache OpenWhisk.

\subsection{Addressed Problem}

The current implementations of the auto-scaling paradigm described above operate under the assumption that \emph{exactly one} server is spawned (or initialized)
when an unfortunate job arrives~\cite{scaleperrequest}.
The objective of this work is \emph{to investigate whether or not it would be convenient to activate more than one server}, say $1+\theta$, instead of just one.
It is worth noting that activating $\theta>0$ extra servers results in increased energy consumption compared to the case where $\theta=0$. On the other hand, this brings a performance benefit \emph{proactively}, as future arrivals have a higher chance of finding active idle servers, thus avoiding the coldstart latency cost.
This is particularly crucial for serverless or edge computing applications, where response time is critical to ensure optimal performance~\cite{DOGANI2024104837}. It is commonly recognized that even a minor rise of even a few milliseconds in latency can have a drastic effect on real-time applications such as e-commerce sales.
Therefore, this paper aims to explore whether the activation of surplus servers at scale-up times ultimately leads to a more favourable balance between energy consumption and user-perceived performance.

\subsection{Stochastic Optimization Framework}

Several analytical performance models have been developed in the literature to evaluate the delay performance and power consumption induced by auto-scaling algorithms; see, e.g., \cite{scaleperrequest,Jonckheere18,Gandhi13}. The starting point of our work is the Markov model proposed in~\cite{scaleperrequest}, which captures the unique details of several existing serverless computing platforms and is also tailored to the auto-scaling algorithm implemented in Amazon's AWS Lambda.
We extend this model to the case where the platform spawns $\theta+1$ servers at the moment of a job arrival if the job finds no active idle server -- the model in \cite{scaleperrequest} is recovered when $\theta=0$.

To find the~$\theta$ that minimizes a cost function that takes into account the blocking probability, i.e., the probability that a random job incurs a coldstart latency, and the energy consumption, we follow a stochastic optimization approach looking for an online learning algorithm.
The main motivation for this approach is that some quantities such as the job arrival rate are time-varying and unknown in advance. However, they are observable and can therefore be learned.
Moreover, the specific structure of the cost function induced by the considered auto-scaling setting brings the additional difficulty that the \emph{gradient} of the cost function is unknown as well, although again observable.
This issue prevents us from relying on the class of Stochastic Gradient Descent (SGD) iterative schemes, which are common in stochastic optimization, and leads us to consider iterative schemes of the Kiefer-Wolfowitz type~\cite{borkar2008stochastic}; see below for further details.

To optimize over $\theta$, an alternative approach would consist in modeling the problem via a Markov decision process (MDP) \cite{Puterman94} and then learning the optimal policy via (variations of) algorithms such as the celebrated UCRL2~\cite{UCRL2}.
In the literature, MDP formulations for problems similar to ours have been developed recently in \cite{IEEE_TR_CS_2024,TournaireHyon2023}; see also the references therein.
In contrast to our approach, these works assume full knowledge of the model parameters and no learning is considered.
%
We do not follow the MDP reinforcement learning approach for two reasons.
In general, the optimal policy is highly dependent on the system state and therefore not versatile.
In our case, however, the optimal policy reduces to a one-dimensional parameter $\theta$, indicating how many servers to activate upon job arrivals. Also,  the size of the state space of the  MDP is prohibitively large  in our case, and this would make any model-based reinforcement learning algorithm impractical. We show in Section~\ref{sec:framework}, that the size of the state space is of the order of~$N^3$ where~$N$ denotes the nominal number of servers that can be up and running, which for several applications is in the order of hundreds or thousands.

{
We also notice that the application of reinforcement learning for autoscaling of serverless applications is currently a bit underexplored~\cite{Benedetti22,Buyya24}.
A Q-learning approach is considered in \cite{Buyya21} and a comprehensive numerical evaluation of existing deep learning algorithms has been recently conducted in~\cite{Buyya24}.
The downside of these works is that no theoretical properties about convergence and optimality are proven.}

Finally, to minimize over $\theta$, a further approach would consist of i) solving the global balance equations of the underlying Markov chain to get the stationary measure induced by a given $\theta$ and then ii) computing the optimal $\theta$ by binary search or relying on standard algorithms for deterministic optimization.
This naive  approach is not interesting either because it requires the knowledge of all the parameters that define the underlying Markov chain. As discussed above, we do not assume this knowledge.

\subsection{Novelty of our Approach: Non-Stationary Samples}
\label{ssec:nonstationary}

When trying to apply existing stochastic optimization techniques to the specific case of auto-scaling, the following technical difficulty appears.
To fully grasp the root of the problem, let us formally introduce the general stochastic optimization framework under investigation.
The objective is to minimize some real function $f(\theta) := \E[F(\theta,X)]$ over $\theta \in \mathbb R^p$, where $X$ is some random variable over $\mathcal{X}$.
The distribution of $X$ and the mapping $F:\mathbb{R}^p \times \mathcal{X}\mapsto \mathbb{R}$ are unknown.
However, it is allowed to get samples $X_1,X_2,\dots$ and, thus, $F(\theta_1;X_1),F(\theta_2;X_2),\ldots$ for different values of $\theta_n$. Here, $F(\theta_n, X_n)$ represents the random cost observed at time step $n$ under the set of parameters $\theta$. To find an optimal $\theta$, one can only rely on such information.
Under certain technical conditions, the Robbins--Monro and Kiefer-Wolfowitz algorithms are the classical iterative schemes that make the sequence $\theta_n$ converge to a minimum of~$f$~\cite{borkar2008stochastic,Rasonyi_Tikosi_2023}; \review{see also~\cite{Walton22}}.
Unfortunately, this type of approach can not be employed within our setting because our problem does not grant access to samples of~$X$. In our case,~$X$ has the stationary distribution of a continuous-time Markov chain that models the dynamics of auto-scaling, and what we can only observe are (non-stationary) samples from the \emph{transient} behavior of such chain;
in practice, this corresponds to collecting observations from the up-and-running real system.
This non-stationarity is the main technical difficulty that singles out our work from existing approaches; for further details, see Section~\ref{ssec:problem} and Remark~\ref{rem:rem1}.
\review{
To deal with this difficulty, we modify the standard Kiefer-Wolfowitz algorithm by introducing a new parameter  that controls how long the system is observed for a given $\theta$ in order to obtain a non-stationary sample that is sufficiently close to the corresponding stationary distribution of the Markov chain parameterized by that given~$\theta$.
While we can prove that the modified algorithm converges a.s. to the optimal value of $\theta$ (Theorem \ref{th:as})  with a state-of-the-art convergence rate in $O(n^{2/3})$ (Theorem \ref{th:rate}), there is still a price to pay for non-stationarity:
\begin{itemize}
\item {\it Additional assumptions: }
Assumption~\ref{as:mixing}, which requires the underlying Markov chain to mix \emph{uniformly}, is critical to our proof technique.
It provides a means to control the accuracy of non-stationary samples and is, to some extent, necessary, as we demonstrate that without it, the desired convergence properties fail to hold. This is supported by numerical evidence.
\item {\it Technical difficulty:} The proof for the convergence rate requires a truncation/extension of the control policy to ensure smoothness of the stationary policy with respect to $\theta$. 
\item  {\it Increased convergence time:}  The convergence rate involves a  term depending on the mixing time, more precisely $\log(1/\rho)$ where $\rho$ is the uniform mixing rate.
\end{itemize}
}

{The closest reference to our work is the classical work~\cite{Pflug1990}, which presents a scheme of the Kiefer-Wolfowitz type as in our setting. Under technical conditions, that scheme converges almost surely to a minimum of the cost function. However, no convergence rate is proven for that scheme.
Another reference that is close to ours is \cite{Borkar22ssy}. Here, the authors consider a general iteration scheme for solving a stochastic optimization problem as in our setting, modulo some minor technical assumptions.
The main differences with respect to our work are that their iterative scheme is of the Robbins-Monroe type and that their main result (Theorem~1) does not specify the convergence speed of the scheme towards the minimum of the cost function. More precisely, they show that it has a polynomial structure, but the exponent depends on a parameter, $\alpha$, related to the Lipschitz constant of the average cost, which may be difficult to get depending on the application considered.}

We also mention a number of related works that have recently appeared in the literature to address settings similar to ours~\cite{MCGD,M1,M2,M3,M4}, i.e., where samples of~$X$ are not available.
These propose SGD methods where approximate samples of~$X$ are taken on the trajectory of a Markov chain, as we do within our approach.
In particular, the algorithm proposed in~\cite{MCGD} has nice convergence properties even in the case of non-convex cost functions and non-reversible Markov chains, which is the setting considered in this paper. The crucial difference between all these works and ours is that they all assume the knowledge of the gradient of the cost function~$f$.
In our case, this information is not available as it depends on the unknown transition rate matrix of the underlying Markov chain.

\subsection{Summary of our Contribution}

First, we model the job and server dynamics induced by our auto-scaling mechanism in terms of a Markov chain that generalizes the one recently proposed in~\cite{scaleperrequest}.
This is parameterized by $\theta$, which defines a set of auto-scaling policies and may be interpreted as a \emph{reserve} of extra servers that are ready to be used --
the model and the algorithm discussed in~\cite{scaleperrequest} are recovered if $\theta=0$.
The parameter $\theta$ is under the control of the system manager and, following the stochastic optimization approach discussed above, the aim is to design an iterative scheme capable of making $\theta$ to converge to $\theta^*$, i.e., the reserve size that minimizes some cost function.
lthough our optimization framework is inspired by auto-scaling, we look at the problem from a larger perspective than auto-scaling and propose a general iterative scheme, see Algorithm~\ref{algo:kw}, which is an adaptation of the celebrated Kiefer-Wolfowitz scheme.
Under natural conditions, in Theorem~\ref{th:as}, we show that the sequence of scalars generated by the proposed algorithm converges almost surely to a minimum of the cost function of interest, and in Theorem~\ref{th2}, we show that the convergence rate is $O(n^{-2/3})$, where $n$ denotes the number of iterations on $\theta$.
%
Finally, we apply the proposed algorithm to the special case of auto-scaling. By means of numerical simulations, we validate the proposed approach and quantify the cost function relative gains with respect to the common scale-up rule where~$\theta=0$.
Within a realistic parametrization of our problem, we show that $\theta^* \approx 6$ and that the proposed algorithm indeed generates a sequence $(\theta_n)_n$ that converges to such a value, yielding relative gains of around 5-8\%.

\subsection{Organization}

This paper is organized as follows.
Section~\ref{sec:framework} introduces a Markov model for the considered auto-scaling system and formalizes the stochastic optimization problem of interest.
Section~\ref{sec:KW} defines our non-stationary Kiefer--Wolfowitz algorithm (Algorithm~\ref{algo:kw}) and presents our main results (Theorems~\ref{th:as} and~\ref{th2}). We stress that this algorithm is general and that it can be applied outside the auto-scaling framework introduced in Section~\ref{sec:framework}.
Finally, Section~\ref{sec:application} is dedicated to the application of Algorithm~\ref{algo:kw} in the context of auto-scaling.
Here, we validate its behavior and evaluate its performance numerically.


\section{Framework}
\label{sec:framework}

\subsection{System Description and Auto-scaling Algorithm}

We consider an architecture composed of $N$ parallel servers; in serverless computing, servers are also called function instances. These represent the nominal service capacity, i.e., the upper limit on the number of servers that one user of the platform can have up and running at the same time.
To ensure service availability for other users, existing serverless platforms require the specification of such limit \cite{scaleperrequest}.

Each server can be in one of the following three macro states: \emph{warm} if turned on, \emph{cold} if turned off, and \emph{initializing} if making the transition from cold to warm.
An initializing server cannot process jobs yet as it performs basic startup operations such as connecting to database, loading libraries, etc.
We also say that a server is \emph{idle-on} if it is warm but not processing any job, and \emph{busy} if it is warm and processing some job.
For our purposes, it is convenient to split the set of initializing servers into two groups, say \emph{init}$_0$ and \emph{init}$_1$.
Init-1 servers are initializing servers that are already bound to a job, i.e., the job that triggered their activation. Upon finishing their initialization phase, they process their associated job immediately.
Init-0 servers are initializing servers that are not necessarily bound to any job. Upon finishing their initialization phase, they become either idle-on or busy depending on whether a  job is blocked in the queue.
Warm servers can make the transition to cold only if they are idle-on.
Figure~\ref{fig:serverstates} summarizes the possible transitions among the server states.
\begin{figure}[h]
 \centering
 \includegraphics[width=\columnwidth]{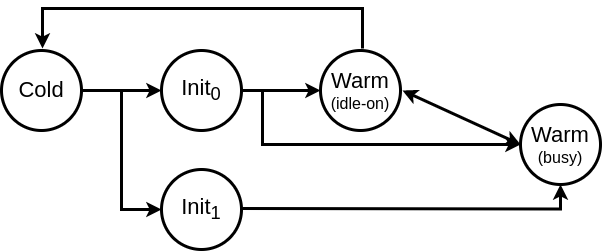}
 \caption{State transitions for each server.}
\label{fig:serverstates}
\end{figure}

Jobs join the system from an exogenous source to receive service.
Upon each job arrival:
\begin{itemize}
 \item if an idle-on server exists, then the job is sent and processed by an idle-on server selected uniformly at random;

 \item if all servers are either busy or init$_1$, all resources are saturated and the job is rejected;

 \item otherwise, a certain number of cold servers
 is selected uniformly at random to become initializing, and in the meanwhile the job waits in a central queue for the activation of its init$_1$ server.
 that is again selected
 upon arrival of the job itself.

\end{itemize}

In the central queue, jobs wait following the first-come first-served scheduling discipline, and each job leaves the system upon completion at its designated server.
We notice that scale-up decisions are taken at job arrival times and by a central monitor, which at any point in time has full knowledge of the server states.
In the literature, this auto-scaling mechanism is called ``scale-per-request''~\cite{scaleperrequest,ASLANPOUR2024266}.
On the other hand, the decision of making a server cold is taken by the server itself after an \emph{expiration time} if during such time the server received no job.
This scale-down rule is well consolidated in practice~\cite{Peeking18,SPEC18,Gandhi13} and it will be assumed in the following.


Since the scale-down rule is fixed, we focus on the design of a scale-up rule.
As commented in the Introduction, current scale-up rules initialize at most one server for each job arrival.
The novelty of our approach consists in providing more flexibility in this respect, and this flexibility is captured by controlling the number of init$_0$ servers.
More precisely, Algorithm~\ref{algo:autoscaling} defines our auto-scaling algorithm.

\begin{algorithm}
\label{algo:autoscaling}
\caption{The proposed auto-scaling algorithm.}
\KwIn{$\theta\in \mathbb{R}_+$, ``system state''}
\KwOut{Number of servers to initialize: \#init$_0$\_servers:=0 and  \#init$_1$\_servers:=0}
\For{each job arrival}{
    \#init$_0$\_servers:=0\\
    \#init$_1$\_servers:=0\\
    \uIf{\#{idle-on}\_servers=0 and \#{cold}\_servers>0}
    {
    \#init$_1$\_servers:=1\\
    \#init$_0$\_servers:=$\pi_\theta$(``system\_state'')\\
    }
 }
\end{algorithm}

Let us elaborate a bit more on Algorithm~\ref{algo:autoscaling}.
First, servers are only initialized if there is no idle-on server.
In this case, the number of init$_1$ servers to activate is always set to one, because by definition this init$_1$ server will be the server that will process the incoming job.
Then, the number of init$_0$ servers to activate is given by the ``black box'' $\Pi_\theta$.
The form of this function is not important for now and will be given in Section~\ref{sec:numerical}.
What should be retained is that it depends on the system state and on the parameter $\theta\in\mathbb{R}_+$, i.e., the design parameter that will be the subject of stochastic optimization.
We anticipate that the idea behind our scale-up rule $\pi_\theta$ will be to choose \#init$_0$\_servers to ensure that an average reserve of $\theta$ servers is ready for use at least in the short future.
Within this interpretation, if $\theta=0$ then no init$_0$ servers exist and Algorithm~\ref{algo:autoscaling} boils down to the auto-scaling algorithm investigated in~\cite{scaleperrequest}.

\subsection{Markov Model}
\label{ssec:Markov_model}

We introduce a continuous-time Markov chain that models the dynamics induced by the
system described above.

We assume that jobs join the system following a Poisson process with rate $\lambda$.
Service, initialization and expiration times are exponentially distributed random variables with rate $\mu$, $\beta$, $\gamma$, respectively.
The four sequences of job inter-arrival, service, initialization and expiration times are i.i.d. and independent of each other.

Let $\mathcal{X}:= \{x = (x_1,x_2,x_3,x_4)\in\mathbb{N}^4: \sum_{i=1}^3x_i\le N, x_4\le x_3\}$.
A Markovian representation of the system dynamics is captured by the state variable $ x \in \mathcal X$
where $x_1$, $x_2$ and $x_3$ represent the number of idle-on, busy and initializing (both init$_0$ and init$_1$) servers, respectively, and $x_4$ represents the number of init$_1$ servers only.
Note that the number of init$_1$ servers can be equivalently interpreted as the number of \emph{blocked} jobs, i.e., the jobs that are waiting in the central queue.
Thus, the number of cold servers in state $x$ is $N- \sum_{i=1}^3 x_i$.

Let $e_i$ be the size-$4$ unit vector in direction $i$.

The Markov chain $(X_t)_t$ that describes the dynamics of the proposed auto-scaling algorithm is defined by the transition matrix
\begin{align}
\label{eq:Q_def}
Q_{x,x'} =
\left\{
\begin{array}{ll}
\lambda \I{x_1>0} &  \mbox{ if } x'=x - e_1 + e_2 \\
\lambda \I{x_1=0} &  \mbox{ if } x'=x + \pi_\theta(x) e_3 + e_4 \\
\mu x_2 &  \mbox{ if } x'=x + (e_1 - e_2)\I{x_4=0} \\ & \qquad\qquad\,\, - \,e_4\, \I{x_4>0} \\
\gamma x_1 & \mbox{ if } x'=x - e_1 \\
\beta x_3 \I{x_4>0} & \mbox{ if } x'=x + e_2-e_3-e_4 \\
\beta x_3 \I{x_4=0} & \mbox{ if } x'=x + e_1-e_3 \\
\end{array}
\right.
\end{align}
for all states $x,x'\in\mathcal{X}$, with $x'\neq x$,
where
$\pi_\theta(x)$ represents the scale-up policy, i.e., the number of servers to initialize upon job arrival and when the system is in state $x$ with $x_1=0$ and $N- \sum_{i=1}^3 x_i>0$.
Again, the precise form of $\pi_\theta(x)$ will be given in Section~\ref{sec:numerical}.

\subsection{Stochastic Optimization Problem}
\label{ssec:problem}


Our objective is to find $\theta$ that minimizes a trade-off between user-perceived performance and energy consumption,
and as discussed in the Introduction, we follow a stochastic optimization approach.
To investigate this trade-off, we first define the instant random cost:
\begin{equation}
 \label{eq:cost}
 C(x) := \sum_{i=1}^4 w_i x_i + \I{ x_2+x_4=N} w_{\rm rej},
\end{equation}
where $w_i$, for $i=1,\ldots,4$, and $w_{\rm rej}$ are positive weights;
here, for instance, $w_{\rm rej}$ is the weight for rejecting a job because $\I{ x_2+x_4=N}$ represents the event that the system contains $N$ jobs or, equivalently, that all the $N$ servers are used.
This cost function is  known to the optimizer. 
However,  we are interested in finding $\theta$ that minimizes the average long-run cost in a setting where \emph{the transition rate matrix $Q$ in \eqref{eq:Q_def} is unknown but the states of the underlying Markov chain,  $(X_t)_t$, can be observed} by the optimizer.
More precisely, we want to develop an iterative scheme able to find $\theta\in\mathbb{R}_+$ that minimizes
\begin{align}
\label{eq:objective}
c(\theta) := \E[C(X_\infty^\theta))]
\end{align}
where $X_\infty^\theta$ is a random variable having the invariant distribution of $(X_t)_t$, which exists because it is ergodic and depends on $\theta$ because of~\eqref{eq:Q_def}.
{In addition, we require that the iterative scheme that we look for should be potentially implemented in a real system. To achieve this goal, the optimizer can only use samples from a {\it single run} of the Markov chain and adjust the value of $\theta$ on the fly.}

\begin{remark}
\label{rem:rem1}
The structures of \eqref{eq:cost} and the cost function described in Section~\ref{ssec:nonstationary} are different in the sense that~\eqref{eq:cost} does not depend on $\theta$ \emph{directly}.
It only depends on~$\theta$ via $X_\infty^\theta$.
{This poses the technical difficulty that $\frac{{\rm d}}{{\rm d}\theta}C$ is not accessible and rules out the application of a large class of stochastic gradient descent algorithms based on stochastic approximations with Markovian noise (see for example \cite{allmeier_gast2024,Borkar22ssy}).}
\end{remark}

\begin{remark}
If one could sample directly from $X_\infty^\theta$, then the optimization problem under investigation could be solved via the celebrated Kiefer--Wolfowitz algorithm which, under certain conditions, can provide a sequence of $\theta$'s converging to a minimum of $c(\theta)$.
However, as discussed in the introduction, we have only access to (transient) samples of $(X_t)_t$ rather than $X_\infty^\theta$ and this makes our problem more difficult.
\end{remark}

\section{A Non-stationary Kiefer--Wolfowitz Algorithm}
\label{sec:KW}

In this section, we propose a general stochastic descent algorithm that solves a general optimization problem.
The relationship with auto-scaling as formulated previously is quite intuitive but it will  be formally  explained  in Section \ref{sec:application}.
Here, we consider a quite general parametric finite Markovian system, with parameter denoted $\theta \in \R$.
The finite state space is denoted by $\mathcal{X}$ and the transition matrix under $\theta$ is $P_\theta$.
The aim is to design an online learning algorithm that computes the optimal value  $\theta^*$, which  minimizes an expected  cost  under the  {\it stationary} regime of the Markov chain.

In the following, we denote by $f$ the function to minimize, defined as
\begin{align*}
  f: \R &\to \R\\
  \theta &\mapsto \E\brac{\cost(\theta,X_\infty^\theta)}
\end{align*}
where $F(\theta,x)$ is the cost under state $x$ and parameter $\theta$. Here, $X_\infty^\theta$ is a  random state whose distribution is  $m_\theta$, the stationary distribution of the Markov chain with parameter $\theta$.

The rest of this section is organized as follows.
Section~\ref{ssec:algo} introduces our online algorithm;
Section~\ref{ssec:results} shows that under some regularity assumptions on the cost, our algorithm  converges to the optimal parameter $\theta^*$ almost surely (Theorem \ref{th:as}). In addition, a state-of-the-art convergence rate in  $O(n^{-2/3})$  is proved in Theorem \ref{th:rate}.

\subsection{Description of the Algorithm}
\label{ssec:algo}

Our non-stationary Kiefer--Wolfowitz (KW) algorithm is based on the following idea.
{After each policy update, from episode $n-1$ to $n$, the classical KW algorithm needs to sample two independent states  from the stationary measures of the Markov chain with parameters $\theta_{n}+\delta_{n}$ and $\theta_{n}-\delta_{n}$, where $\delta_n$ is the stepsize that is needed to approximate the gradient of the cost function at $\theta_n$. However, we do not have access to these stationary measures. Instead, we simulate one  Markov chain over $\tsimul_n$ timesteps twice to reach two  states close to stationarity, where $\tsimul_n$ is related to the mixing time of the Markov chain and scales in $\log n$.}

We start from the initial state $x_{\rm start}$ and initial policy $\theta_0$. Denote by $T_n$ the number of timesteps at the end of episode $n-1$. For each episode $n$, we also choose $x_{\rm start}$ to be the initial state of the trajectories we will simulate.

From $T_n$ to $T_n+\tsimul_n-1$, we first simulate the Markov chain with initial state $x_{\rm start}$ and parameter $\theta_n+\delta_n$, and observe the states $X_{T_n,T_n+i}^{\theta_n+\delta_n}$ for $i=1,\ldots,\tsimul_n$, and the random cost $\cost(\theta_n+\delta_n,X_{T_n,T_n+\tsimul_n}^{\theta_n+\delta_n})$. We reiterate this process $K$ times, and do similar simulations for the Markov chain with parameter $\theta_n-\delta_n$. 
 We then evaluate the average cost of the Markov chain with parameters $\theta_n+\delta_n$ and $\theta_n-\delta_n$ with the samples $X_{T_n+i\tsimul_n,T_n+(i+1)\tsimul_n}^{\theta_n+\delta_n}$ and $X_{T_n+(K+i)\tsimul_n,T_n+(K+i+1)\tsimul_n}^{\theta_n-\delta_n}$ for $i=0,\ldots,K-1$, from which we may approximate the derivative of the cost function at $\theta_n$, and eventually compute the following update of the parameter $\theta$:
\begin{equation}
 \label{eq:update}
 \theta_{n+1}=\theta_n- \stepsize_n \frac{\hat f_n(\theta_n+\delta_n)-\hat f_n(\theta_n-\delta_n)}{2\delta_n},
\end{equation}
where $\stepsize_n$ is the stepsize of the parameter update, and
\begin{footnotesize}
\begin{subequations}
\label{eq:grad_estimates}
\begin{align}
 \hat f_n(\theta_n+\delta_n)&=\frac{1}{K}\sum_{i=0}^{K-1}\cost\prt{\theta_n+\delta_n,X_{T_n+i\tsimul_n,T_n+(i+1)\tsimul_n}^{\theta_n+\delta_n}}, \\
 \hat f_n(\theta_n-\delta_n)&=\frac{1}{K}\sum_{i=0}^{K-1}\cost\prt{\theta_n-\delta_n,X_{T_n+(K+i)\tsimul_n,T_n+(K+i+1)\tsimul_n}^{\theta_n-\delta_n}}.
 \end{align}
\end{subequations}
\end{footnotesize}
This process is formalized in Algorithm~\ref{algo:kw}.

\begin{algorithm}
\label{algo:kw}
  \caption{Non-stationary Gradient Descent Algorithm.}
  \KwIn{$\stepsize_n$ step-size sequence, $\delta_n$ discretization step, initial parameter $\theta_0$ and state $x_{\rm start}$, $T$ the total simulation time and a parameter $\tsimul$.}
  Set $n=0$ the algorithm time-step and $t=0$ the current simulation time-step, $x=x_{\rm start}$ the initial state.\\

  \While{$t\leq T$ }{
  $\tsimul_n=\tsimul \log (n+1)$ \\
    \For{ the simulation number $i=0,1,\ldots,K-1$}{
    Simulate the Markov chain starting at $x_{\rm start}$ with parameter $\theta_n+\delta_n$ over $\tsimul_n$ steps, by choosing at each step of the simulation the action $\hat \theta$ randomly between $\lfloor \theta_n+\delta_n \rfloor$ and $\lfloor \theta_n+\delta_n \rfloor+1$.\\
    Repeat this process overall $K$ times, and observe the empirical average reward and the average visit count for each state.\\
    Repeat this process for the parameter $\theta_n-\delta_n$.\\
    }
  Compute the average over the $K$ simulations to obtain $\hat{f}_n (\theta_n+\delta_n)$ and $\hat{f}_n(\theta_n-\delta_n)$.\\
  Update the empirical stationary measure: the number of visits of a given state under any parameter, divided by the total number of visits.\\
  Compute the parameter update \eqref{eq:update}:
  $\theta_{n+1}=\theta_n-\stepsize_n\frac{\hat {f} (\theta_n+\delta_n)-\hat{f}(\theta_n-\delta_n)}{2\delta_n}.$\\
  $t:=t+2K\tsimul_n$ and  $n:=n+1$
}
\end{algorithm}

{In the loop in Line 4, we simulate the trajectory induced by $\theta_n+\delta_n$ before the one induced by $\theta_n-\delta_n$. Because of the memoryless property, this does not affect our results. }

\subsection{Convergence Results}
\label{ssec:results}

We will use the following assumptions.

\begin{assumption}[Regularity of the cost function]
\label{as:reg}
\ \\
\vspace{-0.5cm}
\begin{enumerate}
 \item[(1.a)] $f:\R \to \R \in \mathcal C ^3$,
 \item[(1.b)] $f'$ is Lipschitz,
\item[(1.c)] the cost function $\cost$ can be written as $\cost(\theta,x)=\costOg(\theta,x)+\pen(\theta)$, where $\pen$ is a penalty function such that $\abs{\pen(\theta+\varepsilon)-\pen(\theta-\varepsilon) }\leq C_2(\theta-\theta^*)\varepsilon$ for $0<\varepsilon<\varepsilon_0$, where $C_2$ and $\varepsilon_0$ are positive constant, and $\costOg$ is positive and bounded by $\costOg_{\max}$ and $\theta^*$ is a minimum of $f$.
\item[(1.d)] There exists $L>0$ and $r>0$ such that
$f'(\theta) > r$ for $\theta>L$, and $f'(\theta) < -r$ for $\theta<-L$.
\end{enumerate}
\end{assumption}

{
While Assumption~\ref{as:reg} is quite technical and may look restrictive, it is satisfied in the autoscaling case by simple inspection of the cost function.
This will be shown in Section~\ref{sec:application}.}



\begin{assumption}[Uniform mixing time]
 \label{as:mixing}
Let $P_\theta$ be the transition matrix of a Markov chain indexed by the parameter $\theta\in\mathbb{R}$.
Let $x$ be an initial state and $m_\theta$ denote the corresponding stationary measure within parameter $\theta$.
There exists $C_1>0$ and $\rho<1$ independent of $\theta$ such that
$$\left\|P_\theta^t(x,\cdot)-m_\theta \right\|_1\leq C_1\rho^t$$
for any $\theta$ and $t$.
\end{assumption}
The previous assumption is not common in stochastic gradient descent algorithms because standard approaches assume to be able to sample from a stationary distribution. Indeed, in our framework it relates the transient regime that we can observe over time to stationary properties of the system.

\begin{assumption}[Uniqueness]
\label{as:uniqueness}
The function $f$ has a unique minimum $\theta^*$.
\end{assumption}

\begin{assumption}[Strong convexity]
\label{as:convex}
The function $f$ is strongly convex: for some $\kappa >0$, for any $\theta,\theta' \in \R^+$, it holds that
\begin{equation}
 \label{eq:convex}
 f'(\theta)(\theta-\theta') \geq \kappa (\theta-\theta')^2.
\end{equation}
\end{assumption}

We can now state our main results.

The next theorem states the almost sure convergence of the sequence $\theta_n$ produced by the proposed algorithm, under the following parametrization:
The sequences $(\stepsize_n)_n$, $(\delta_n)_n$ and $(\tsimul_n)_n$ are such that:
\begin{subequations}
\label{eq:parameter1}
\begin{align}
&\lim_{n\to\infty} \delta_n=0,
&\lim_{n\to \infty} \tsimul_n= +\infty\\
&\sum_{n} \stepsize_n = \infty,
&\sum_{n}  \stepsize_n^2 \delta_n^{-2}   < \infty
\end{align}
\end{subequations}
and
\begin{equation}
\label{eq:parameter2}
(\stepsize_n)_n, (\delta_n)_n \text{ and } \prt{\frac{\stepsize_n}{\delta_n}}_n \text{ are decreasing.}
\end{equation}

\begin{theorem}
 \label{th:as}
 Let $(\theta_n)_n$ be the sequence of random variables generated by Algorithm~\ref{algo:kw} with parametrization \eqref{eq:parameter1}-\eqref{eq:parameter2}.
 Under Assumptions~\ref{as:reg},~\ref{as:mixing}, and ~\ref{as:uniqueness}, we have
 \begin{equation}
  \theta_n \overset{a.s.}{\to} \theta^*.
 \end{equation}
\end{theorem}

The next theorem provides a result on the convergence rate to the minimizer of~$f$.

\begin{theorem}
  \label{th2}
  \label{th:rate}
Let Assumptions~\ref{as:reg},~\ref{as:mixing},~\ref{as:uniqueness} and~\ref{as:convex} hold.
Under well-chosen parameters $\delta_n=n^{-2/3}$, $\stepsize_n=n^{-1}$ and $T_n=\alpha\frac{\log n}{\log 1/\rho}$ with $\alpha>1$, and $\stepsize_0 < \frac{4 \kappa}{C_2}$, Algorithm~\ref{algo:kw} converges to the minimum $\theta^*$ with asymptotic rate:
\begin{multline*}
\limsup_{n \to \infty} \E\brac{(\theta_n-\theta^*)^2} n^{{2/3}}
\\
\leq \frac{\prt{2C_0+\sqrt{2}\costOg_{\max}C_1^{1/2}}^2}{8\kappa^2}+\frac{\costOg_{\max}^2}{2\kappa}.
\end{multline*}
\end{theorem}

{
  Theorem~\ref{th2} implies that the convergence rate of the sequence $\theta_n$ produced by Algorithm~\ref{algo:kw} is $O(n^{-2/3})$, provided that its input parameters are properly tuned, which is as good as the state-of-the-art convergence rate of classical KW algorithms~\cite{Walton22}.
  The detailed proof  is postponed to Appendix~\ref{sec:proof}.
}

\section{Application to Auto-scaling}\label{sec:application}

In this section, we discuss the applicability of Theorems \ref{th:as} and \ref{th:rate} to the auto-scaling problem modeled in Section~\ref{sec:framework}.
The main point is to construct a cost function $f$ satisfying  Assumptions \ref{as:reg} whose  minimum coincides with the minimum of the function $c$ defined in \eqref{eq:objective}.

This construction is not unique and several choices made in the following could certainly be changed, especially to improve the performance of the algorithm in practice (see Section \ref{sec:numerical}).

\subsection{Truncation/Extension: Construction of the Scale-up Rule}
\label{ssec:extension}

The first step is to construct a scale-up rule $\pi_\theta$ that maps any real parameter $\theta$ into a number of servers in $[0,N]$.
The idea is as follows: see $\theta$ as the average amount of servers to turn on. A simple choice would be  to sample $\pi_\theta$ from a binomial law with parameters $(N,\theta/N)$. This would be possible for $\theta \in (0,N)$. However, to comply with the optimization algorithm given in Section \ref{sec:KW},  $\theta$ must live on the whole real space $\R$.

For that, we construct an explicit mapping from $\R$ to $[\varepsilon, M-\varepsilon]$, with $\varepsilon > 0$, and $M <N$.
Here, $\varepsilon$ must be seen as a very small parameter whose role is only to get a smooth transition from $\R$ to $[0,M]$.
As for the choice of $M<N$, it  will be explained in the next subsection.

%
%
%
%
First, define the following smooth step function $\psi_{a,b}$, for $a<b \in \R$:
\begin{equation*}
 \psi_{a,b}: x \in \R \mapsto
 \begin{cases}
  0 & \text{ if } x < a \\
  \exp \prt{-\frac{(b-x)^2}{x-a}} & \text{ if } a \leq x \leq b \\
  1 & \text{ if } b \leq x.
 \end{cases}
\end{equation*}
This function is equal to $0$ on $(-\infty,a]$, equal to $1$ on $[b,+\infty)$, and is smooth on $\R$. Using this function, for $\varepsilon >0$,
we define
\begin{equation*}
 \theta_{\varepsilon,M} :\theta \in \R \mapsto
 \begin{cases}
   h_-(\theta) , \quad \text{ if } \theta < 0 \\
  h_-(\theta) \prt{1-\psi_{0,\varepsilon}(\theta)}+\theta \psi_{0,\varepsilon}(\theta),\\
  \quad \text{ if } 0 \leq \theta \leq \varepsilon \\
  \theta,
  \quad \text{ if } \varepsilon < \theta < M-\varepsilon \\
  \theta \prt{1-\psi_{M-\varepsilon,M}(\theta)}+h_+(\theta) \psi_{M-\varepsilon,M}(\theta),
  \\ \quad \text{ if } M-\varepsilon \leq \theta \leq M \\
  h_+(\theta), \quad \text{ if } M < \theta,
 \end{cases}
\end{equation*}
where $h_-(\theta)=\frac{\varepsilon}{3} \exp\prt{\frac{\theta}{\varepsilon}}$ and $h_+(\theta)=M-\frac{\varepsilon}{3} \exp\prt{- \frac{\theta-M}{\varepsilon}}$. We display a representation of this function for $M=10, \varepsilon=0.5$ in Figure~\ref{fig:parameter_function}.
\begin{figure}[hbtp]
 \centering
 \includegraphics[scale=0.60]{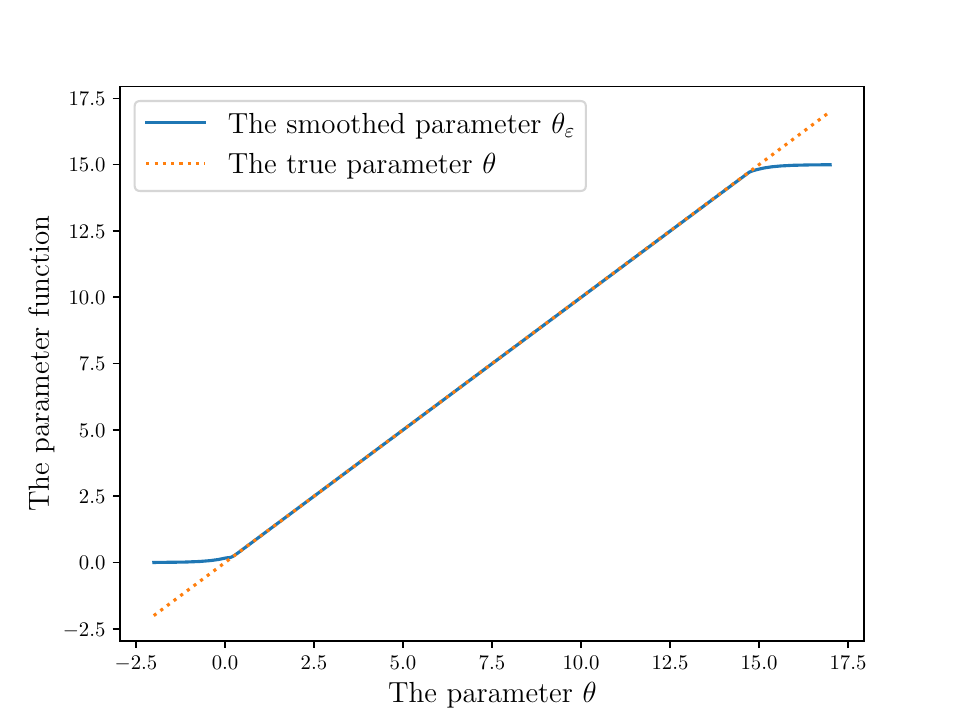}
 \caption{The smoothed parameter function $\theta_{\varepsilon,M}$ compared with $\theta$ \label{fig:parameter_function} }
\end{figure}
Here, $\theta_{\varepsilon,M}$ is adapted from $\theta$ and has the following properties:

\begin{itemize}
\item $\theta_{\varepsilon,M} \in [0,M]$,
 \item For $\theta \in [\varepsilon, M-\varepsilon]$, $\theta_{\varepsilon,M} = \theta$,
 \item $\theta_{\varepsilon,M} \to 0$ as $\theta \to -\infty$,
 \item $\theta_{\varepsilon,M} \to M$ as $\theta \to +\infty$.
\end{itemize}
We will therefore be able to sample from a binomial law with parameters $(M,\theta_{\varepsilon,M} /M)$. At any state $x$, the number of servers to turn on $\pi(x)$ follows a binomial law with parameters $(M,\theta_{\varepsilon,M})$.



\subsection{Smoothness}


Since the mapping $\theta \to \theta_{\varepsilon,M} \to  P_{\theta_{\varepsilon,M}}$ defined in the previous subsection only constructs  irreducible matrices, the set of recurrent states of the Markov chain  remains unchanged as $\theta$ gets updated. This implies that
the rank of $P_{\theta_{\varepsilon,M}}$ is constant for all $\theta \in \R$.

By construction, we also have for all  $\theta \in \R$, $\theta \mapsto P_{\theta_{\varepsilon,M}} \in \mathcal C^3$ (actually, it is in  $C^k$ for all $k> 0$).

These two properties imply that the stationary measure is also smooth, which we will prove by proving regularity properties of the pseudoinverse matrices related ot the transition matrices $P_{\theta_{\varepsilon,M}}$.

 The Drazin inverse of a matrix  is defined as follows.
\begin{definition}[\cite{Cline68}]
  \label{def:drazin}
  Let $A$ be a square real-valued matrix. The Drazin inverse exists and is the unique solution $X$ to the following equations:
  \begin{equation*}
   \begin{cases}
    AX &= XA, \\
    A^k &= A^{k+1}X \quad \text{ for some positive integer $k$,}\\
    X &=X^2 A.
    \end{cases}
  \end{equation*}
The smaller $k$ for which these equations still hold is called the Drazin index of $A$. We will denote the Drazin inverse of $A$ by $A^\#$.
 \end{definition}
We will be able to relate $f$ to the Moore-Penrose pseudo inverse of a matrix,
also called generalized inverse, which will be denoted by the superscript $^\dag$.
Regularity properties of this pseudo inverse will be proven using the following lemmas.
 \begin{lemma}[Theorem 4.3, \cite{Golub73}]
  \label{lem:diff_pseudo}
  Let $\theta \mapsto A_\theta$ be a Fréchet differentiable square matrix function with local constant rank in $\R$. We denote by $\mathbf{D}A_\theta$ the Fréchet-derivative at $\theta$ and we write the following equation as a function from $\R$ to $\R^{|\mathcal X|\times |\mathcal X|}$. For any $\theta \in \R$:
  \begin{equation}
  \label{eq:diff_pseudo}
  \mathbf{D}A_\theta^\dag = -A_\theta^\dag\mathbf{D}A_\theta A_\theta^\dag + A_\theta^\dag A_\theta^{\dag \top} \mathbf{D}A_\theta^\top P_{A_\theta}^\bot + \hspace{0.0cm} _{A_\theta}\hspace{-0.02cm} P^\bot \mathbf{D} A_\theta^\top A_\theta^{\dag \top} A_\theta^\dag,
 \end{equation}
 where $P_{A_\theta}^\bot:=I-A_\theta A_\theta^\dag$ is the projector on the orthogonal complement of the space spanned by the columns of $A_\theta$, and $_{A_\theta}\hspace{-0.02cm} P^\bot := I-A_\theta^\dag A_\theta$ is similarly defined for the space spanned by the rows of $A_\theta$.
 \end{lemma}

 In the previous lemma, the Fréchet derivatives can be directly seen as matrix functions $ \R \to \R^{\abs{\mathcal X} \times \abs{\mathcal X}}$ rather than linear applications.

We can now state and prove the following regularity property on the cost function.
\begin{proposition}
 \label{pro:regularity}
 The function $\theta \mapsto  m_{\theta_{\varepsilon,M}}(x)$ is $C^\infty$.
\end{proposition}
\begin{proof}
 Let $x \in \mathcal X$ be any state. Consider a Markov reward process with rewards $r(x)=1$ and $r(y)=0$ for $y \neq x$, with the same transitions $P_{\theta_{\varepsilon,M}}$. This Markov chain is unichain and the gain at $x$ is equal to $m_\theta(x)$, the stationary measure at $x$. It can be written with the Ces\`aro-limit:
 $$m_\theta(x) = \mathbf{e}_x^\top {P_{\theta_{\varepsilon,M}}^\infty} \mathbf{e}_x,$$
 where $P_{\theta_{\varepsilon,M}}^\infty:= \lim_{T \to \infty} \frac{1}{T}  \sum_{t=1}^T P_{\theta_{\varepsilon,M}}^t$ is the limiting matrix and $\mathbf e_x$ is the vector equal to $1$ at $x$ and $0$ everywhere else.
 With \cite[Theorem A.7]{Puterman94}, we relate the limiting matrix $P_\theta^\infty$ to the Drazin inverse $(I-P_{\theta_{\varepsilon,M}})^{\#}$ of $(I-P_{\theta_{\varepsilon,M}})$. Letting $I$ denote the identity matrix:
 $$ P_{\theta_{\varepsilon,M}}^\infty=I - (I-P_{\theta_{\varepsilon,M}})(I-P_{\theta_{\varepsilon,M}})^{\#}.$$
 Let us now call $A_\theta:=I-P_{\theta_{\varepsilon,M}}$, and denote by $k$ the Drazin index of $A_\theta$, as defined in Definition~\ref{def:drazin}. We use \cite[Theorem 5]{Cline68} to relate the Drazin inverse to the generalized pseudoinverse in the following way:
 \begin{equation}
  \label{eq:drazin_to_pseudo}
  A_\theta^{\#} = A_\theta^k (A_\theta^{2k+1})^\dag A_\theta^k.
 \end{equation}
 In order to prove the regularity of $f$, we therefore need to show that the function $\theta \mapsto A_\theta^\dag$ is in $C^3$ itself.
 To prove it, recalling that $P_\theta$ and therefore $A_\theta$ is of constant rank over $\R$, we can use Lemma~\ref{lem:diff_pseudo} to explicitly get the derivative of $\theta \mapsto A_\theta$, which itself is continuous, using the continuity of the generalized inverse (\cite[Theorem 5.2]{Stewart69}), as the rank of $A_\theta$ is constant. We can then use equation~\ref{eq:diff_pseudo} to iterate the derivation process and prove that all the derivatives of $\theta \mapsto A_\theta^\dag$ exist and are continuous. Using equation~\ref{eq:drazin_to_pseudo}, we finally get that $\theta \mapsto A_\theta^\#$ is in $C^\infty$, and therefore $\theta \mapsto m_{\theta_{\varepsilon,M}}(x)$~is~$C^\infty$.
\end{proof}

\subsection{Truncation of the Parameter Domain and Penalty Function}
\label{ssec:trunc}


We are now ready to construct the functions $f$  from the autoscaling costs $c$ and $C$.
Let $M <N$ be any truncation of the parameter space.

Set $f_1(\theta) = \E C(X_\infty^{\theta_{\varepsilon,M}})$.
 We can thus rewrite this function as
 \begin{equation}
  \label{eq:recost}
  f_1(\theta)= \sum_{x \in \mathcal X} C(x) m_{\theta_{\varepsilon,M}}(x)
 \end{equation}

 This function is defined on $\R$ and is  $C^\infty$.
 In the following,  we assume that there exists $M < N$ such that this smooth function has a unique minimum inside $(0,M)$ and is strongly convex in $(0,M)$. Such an interval is displayed in red in Figure~\ref{fig:numerical} (left).

The last  step is to construct a penalty function to extend the cost function unboundly outside $[0,M]$.
Let
\begin{equation}
 \label{eq:pen_def}
  \pen(\theta):= (1-\psi_{0,\varepsilon}(\theta))(\theta-\varepsilon)^2 + \psi_{M-\varepsilon,M}(\theta)(\theta-M+\varepsilon)^2.
\end{equation}

We can now construct the function $f$ as follows:
$f(\theta) := f_1(\theta) + b(\theta)$.

\begin{itemize}
\item By construction of the smooth junction between $f_1$ and $b$, $f$ is also  $C^\infty$ over $\R$ so it satisfies Assumption (\ref{as:reg}.a).
By construction, the cost function $f_1$ is $\mathcal C^3$ on the compact $[0,M]$. This implies that $f'_1$ is Lipschitz over $[0,M]$ and bounded over this interval. Since the penalty is quadratic outside $[0,M]$,  $f'$ is also Lipschitz over $\R$. (Assumption (\ref{as:reg}.b))
The decomposition of $f$ into $f_1$ and $b$ also implies Assumption (\ref{as:reg}.c) Assumption (\ref{as:reg}.d), as $f_1'$ is bounded.

\item As for Assumption \ref{as:mixing}, this is a direct consequence of the fact that $\theta_{\varepsilon,M}$ lives in the compact $[0,M]$ and $P^t_\theta$ as well as $m_\theta$ are continuous functions of $\theta$.

\item Finally, Assumptions \ref{as:uniqueness} and \ref{as:convex} are true for  $f$ as soon as one can find a small enough interval $[0,M]$ where $f_1$ is strongly convex and has a unque minimum in the interior of this interval. This is checked  numerically in the following  section.

\end{itemize}

The final point is to notice that $c$ and $f$ coincide over $[\varepsilon, M-\varepsilon]$ and therefore have the same minimum $\theta^*$ in this interval.

\subsection{Numerical Evaluation}
\label{sec:numerical}

By means of numerical simulations, we now evaluate the convergence properties of Algorithm~\ref{algo:kw} when applied to the proposed auto-scaling algorithm (Algorithm~\ref{algo:autoscaling}).
Here, we use a simplified form (compared with Section \ref{ssec:extension}) for the scale-up rule $\pi_\theta(x)$  and we set $\pi(x)$ to 
\begin{align}
\label{def:autoscaling}
%
%
\min\big\{ (\lfloor \theta \rfloor+\I{V< \theta-\lfloor \theta \rfloor} - x_3+x_4)^+,\, N-x_2-x_3-1\big\}
\end{align}
where $V$ is an independent random variable uniformly distributed over [0,1].
Roughly speaking, since the number of init$_0$ servers (i.e., $x_3-x_4$) is the number of servers that will be available to use in the short future, we let this number be $\theta-(x_3-x_4)$.
The randomization in $V$ is used because the number of servers to activate must by an integer but we allow $\theta$ to be a real number because the optimal number of servers to initialize may not be an integer \emph{in average}.
Finally, the $\min$ and $(\cdot)^+:=\max\{\cdot,0\}$ operators simply ensure that boundary conditions are satisfied.

\begin{remark}
If $\theta=0$, then no init$_0$ servers exist, which implies $x_4=x_3$ at all times, and the Markov chain under investigation (defined by \eqref{eq:Q_def}) boils down to the Markov chain investigated in~\cite{scaleperrequest}.
\end{remark}

In our simulations, we consider the following parametrization:
\begin{enumerate}
 \item
For the parameters that define the underlying Markov chain, we let $N=50$, $\lambda\in\{0.15,0.3\}$ (arrival rate), $\mu=1$ (service rate), $\beta=0.1$ (initialization rate) and $\gamma=0.01$ (expiration rate).
If a time unit is 10 milliseconds, these values are realistic for serverless applications~\cite{scaleperrequest,Google}.
 \item
For the parameters that define the cost function \eqref{eq:cost}, we let $w_{\rm rej}=10^3$, $w_1=w_2=1$, $w_3=5$ and $w_4=100$. Note that $w_3\ge \max\{w_1,w_2\}$ because initializing servers perform a batch of operations (connecting to database, loading libraries and data, etc.) and these are very expensive from the point of view of power consumption.
Also, to make user-perceived performance and energy consumption comparable, we choose $w_4$ and $w_{\rm rej}$ to be significantly greater than $w_3$.
 \item
For the parameters that are used by Algorithm~\ref{algo:kw}, in accordance with the assumptions in Theorems~\ref{th:as} and~\ref{th2}, we choose
$\gamma_n=10/n$, $\delta_n=n^{-2/3}$, $\tau = 10^6$, $\theta_0\in\{1,10\}$, $K=2$ and $T=10^8$.
\end{enumerate}

\begin{figure*}
\centering
\makebox[\textwidth][c]{\hspace{0cm}\includegraphics[width=1.0\textwidth]{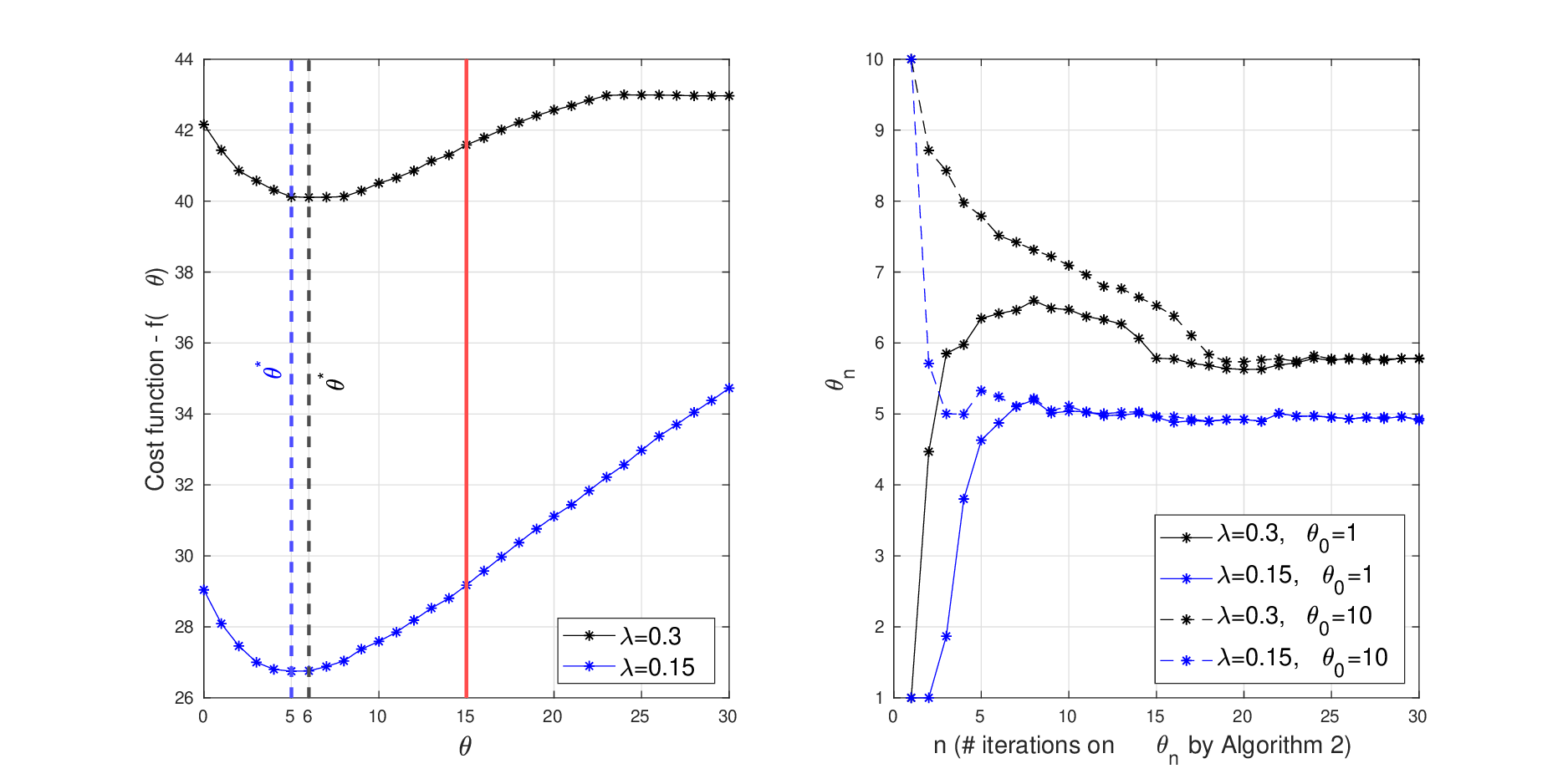}}%
\caption{Plots of the cost function $f(\theta)$ (left) and of the sequences produced by Algorithm~\ref{algo:kw} (right).
The red line corresponds to the truncation to the interval $[0,M]$ over which the function is convex.}
\label{fig:numerical}
\end{figure*}

Figure~\ref{fig:numerical} (left) plots the cost function $c(\theta)$ (defined in \eqref{eq:objective}), which plays the role of $f(\theta)$ in Section~\ref{sec:KW}.
It admits a unique minimum, $\theta^*$, and it is strongly convex on a neighbourhood $\theta^*$; here, the red vertical line is used to mark the convexity region postulated in Section~\ref{ssec:trunc}.
These properties have also been tested numerically over a wide range of parameters, and this is in agreement with the approach discussed in Section~\ref{ssec:trunc}.
For the plots in the figure, $\theta^*\in[5,6]$, which means that initializing 6 or 7 servers is much better than initializing 1 server as in~\cite{scaleperrequest}; recall that $\theta$ is the number of \emph{extra} server to initialize.
From the figure, we observe the potential gain in activating the optimal surplus of servers amounts to 5-8\%.

Figure~\ref{fig:numerical} (right) plots the sequence $\theta_n$ produced by Algorithm~\ref{algo:kw} from different initial conditions. All trajectories converge to $\theta^*$.
Thus, the proposed algorithm improves with respect to the existing scale-up rule of no activating extra servers.
Importantly,
we notice that the trajectories corresponding to $\lambda=0.15$ converge much faster than the trajectories corresponding to $\lambda=0.3$. This quantifies the impact of the mixing time of the underlying Markov chain, which in our framework is captured by Assumption~\ref{as:mixing}: the smaller the $\lambda$, the smaller the mixing time and therefore the easier the sampling close to stationarity.

{

\subsection{Comparing with Fast $\theta$-updates}
\label{sec:comparison}

In Algorithm~\ref{algo:kw}, the reserve size parameter $\theta$ is updated only after the system has been observed for a sufficiently long period, ensuring that its dynamics approximate its stationary behavior. This observation time, linked to the mixing time of the underlying Markov chain, is controlled by the variable $\tau_n$. An alternative approach could involve updating $\theta$ on the same timescale as the Markov chain, ``without waiting for stationarity''.
In other words, $\theta$ is be updated after the system undergoes a fixed number of state transitions, such as $10^2$ or $10^3$, since the last update.
As discussed in the Introduction, this stochastic-approximation approach has been already considered in the literature, e.g.,~\cite{allmeier_gast2024}, though we stress that the existing theoretical results do not apply to our case. Thus, we cannot expect that this alternative approach makes $\theta$ converge to the desired optimal point~$\theta^*$.

The goal of this section is to evaluate this alternative approach within the same simulation setting described in Section~\ref{sec:numerical} and compare it with ours.
Within the considered auto-scaling framework, we demonstrate that our approach provides improved performance.

In our simulations, we consider the following setup:
\begin{itemize}
 \item \emph{Scenario 1}: Identical parameter setting. Simulations are performed in exactly the same setting that produced Figure~\ref{fig:numerical} (right) with the exception that $\theta$ is updated every $10^2$ or $10^3$ state transitions of the underlying Markov chain.

 \item \emph{Scenario 2}: Corrected $\gamma$ weights. As in Scenario 1 but~$\theta$ is less sensitive to the gradient updates. Specifically, the weights $\gamma_n$ are multiplied by the number of state transitions ($10^2$ or $10^3$) and divided by $\tau=10^6$ (used to generate Figure~\ref{fig:numerical} (right)). In this manner, $\theta$ changes slower than in Scenario~1 and its variations are of the same order of the ones considered in the evaluation of our algorithm.

\end{itemize}

For Scenario~1, the corresponding sequences of $\theta$'s are plotted in Figures~\ref{fig:comparison1} and~\ref{fig:comparison2}.
\begin{remark}
\label{rem:sametime}
%
Any point of the $x$-axis of Figure~\ref{fig:comparison1} has a corresponding simulation time of the underlying Markov chains, and the same holds true for the $x$-axis of Figure~\ref{fig:numerical} (right).
It is important to remark that such points coincide on a 1:1 scale. In other words, the curves in these figures describe the evolution of $\theta$ on the same time interval.
\end{remark}

\begin{figure*}
\centering
\makebox[\textwidth][c]{\hspace{0cm}\includegraphics[width=1.0\textwidth]{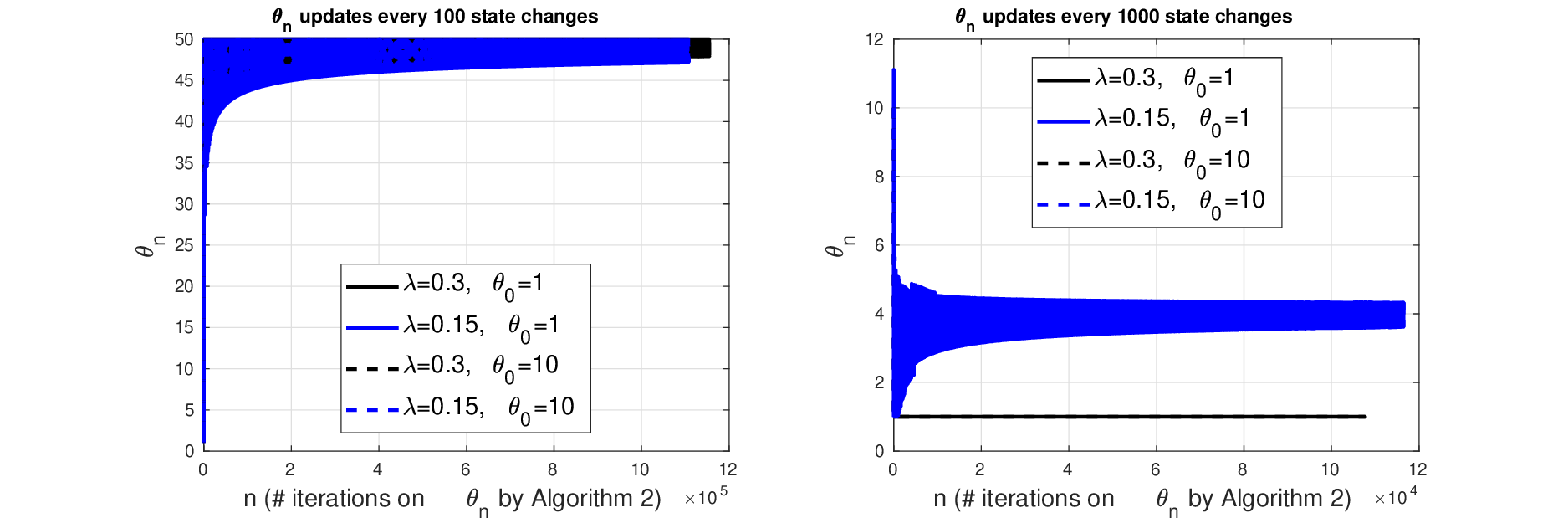}}%
\caption{Scenario 1: Plots of the sequences produced by Algorithm~\ref{algo:kw} when the underlying Markov chain is simulated for $\tau_n=$100 (left) and $\tau_n=$1000 (right) steps.
In simulation time, the $x$-axis corresponds exactly to the one in Figure~\ref{fig:numerical}.
}
\label{fig:comparison1}
\end{figure*}

\begin{figure*}
\centering
\makebox[\textwidth][c]{\hspace{0cm}\includegraphics[width=1.0\textwidth]{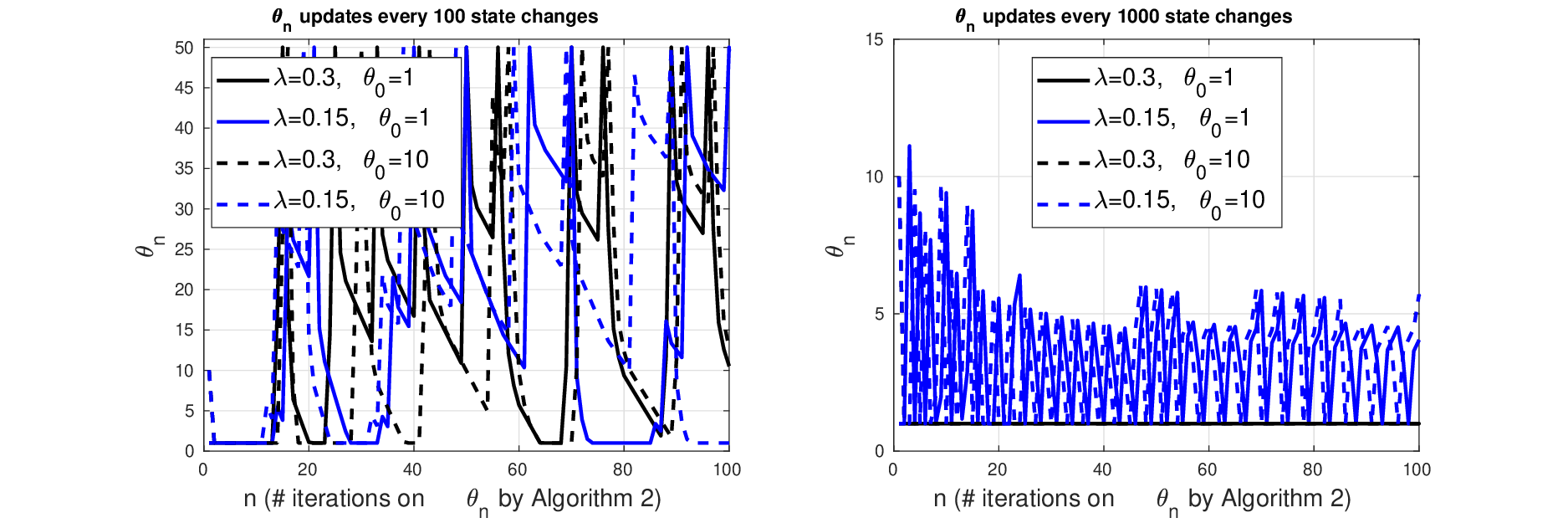}}%
\caption{A zoom of the plots of Figure~\ref{fig:comparison1} obtained by truncation of the $x$-axis.}
\label{fig:comparison2}
\end{figure*}
Let us comment on the plots in these figures:
\begin{itemize}
 \item
As expected, the plots in Figure~\ref{fig:comparison1} are not fully visible
because the fluctuations of $\theta$ are very sensitive to the contribution of the gradient values.
For this reason, we reported the first 100 $\theta$ points in Figure~\ref{fig:comparison2}.

\item The curves in Figure~\ref{fig:comparison1} (left) indicate that $\theta$ tends to follow trajectories that attempt to escape the feasible region from above, which is constrained by $N=50$ servers, implying $\theta \leq 50$ necessarily. These trajectories exhibit oscillations with maximum amplitude ($N=50$), which is impractical and clearly suboptimal.

\item In contrast, for $\lambda=0.3$, Figure~\ref{fig:comparison1} (right) shows that $\theta$ tends to escape the feasible region from \emph{below}. For $\lambda=0.15$, $\theta$ oscillates within the interval $4\pm0.5$, while the optimal $\theta$ (between five and six, as shown in Figure~\ref{fig:numerical}) lies outside this range. While waiting for 1000 state transitions slightly improves results compared to 100 transitions, the outcomes remain unsatisfactory.

\end{itemize}

For Scenario~2, the corresponding sequences of $\theta$'s are plotted in Figure~\ref{fig:comparison3}, where again the $x$-axis maps in simulation time 1:1 with the
the $x$-axis in Figures~\ref{fig:numerical} (right) and~\ref{fig:comparison1}.
\begin{figure*}
\centering
\makebox[\textwidth][c]{\hspace{0cm}\includegraphics[width=1.0\textwidth]{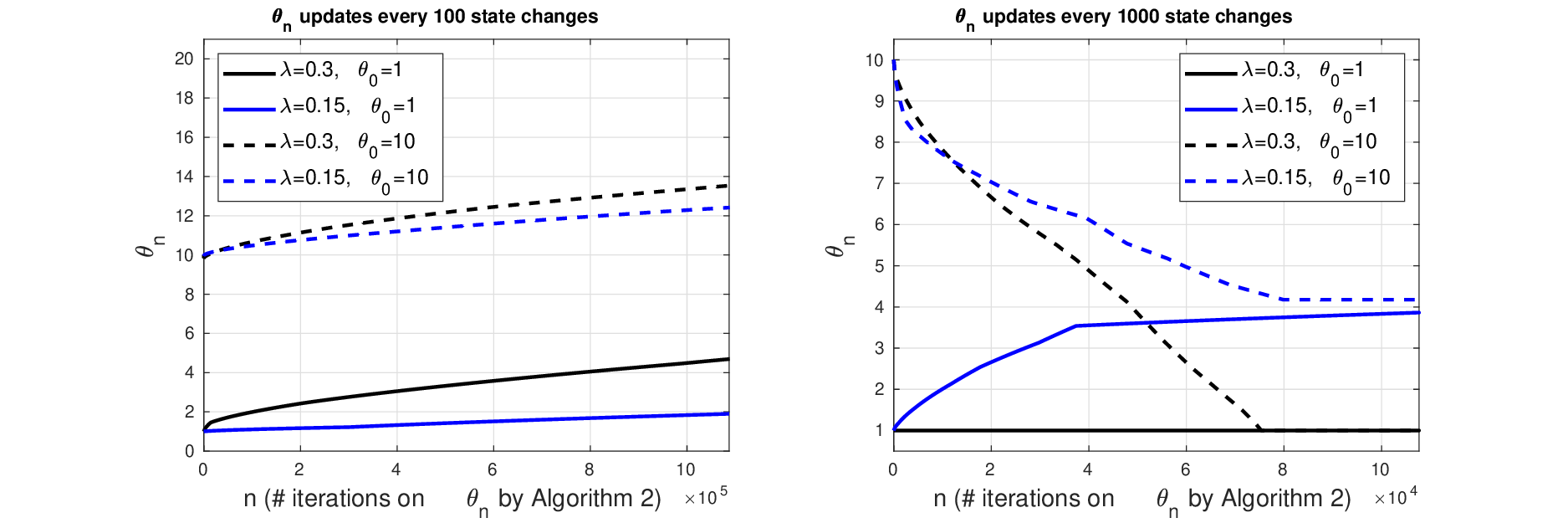}}%
\caption{Scenario 2: Plots of the sequences produced by Algorithm~\ref{algo:kw} when the underlying Markov chain is simulated for $\tau_n=$100 (left) and $\tau_n=$1000 (right) steps.
In simulation time, the $x$-axis corresponds exactly to the one in Figure~\ref{fig:numerical}.
}
\label{fig:comparison3}
\end{figure*}
{As expected, the large oscillations encountered in Scenario~1 are smoothed out by scaling~$\gamma_n$.
We observe that the resulting curves
are sensitive to the initial conditions and input parameters and 
do not  converge to a point: They
tend to move away from the true optimum $\theta^*$, which is around 5 and 6 as shown in Figure~\ref{fig:numerical} (left).
Also, the black curves tend to behave as in Scenario~1 (in average).

In the considered auto-scaling context, the simulation results above indicate that ``the fast $\theta$-update approach'' does not work.}

}




\appendix

\section{Proof of our Main Results}
\label{sec:proof}

In this appendix, we provide proofs for our main results, i.e., Theorems~\ref{th:as} and \ref{th2}.
To do so, it is convenient to rewrite the parameter update rule as follows
 \begin{equation}
  \label{eq:algoRew2}
  \theta_{n+1}=\theta_n-\stepsize_n \prt{f'(\theta_n) + \btermn + \ntermn + \mtermn},
 \end{equation}
where we define
\begin{align*}
 \ntermn:=&\frac{\cost(\theta_n+\delta_n,X_\infty^{\theta_n+\delta_n})-\cost(\theta_n-\delta_n,X_\infty^{\theta_n-\delta_n})}{2\delta_n} \\
 & -\frac{f(\theta_n+\delta_n)-f(\theta_n-\delta_n)}{2\delta_n},\\
 \btermn:=&\frac{f(\theta_n+\delta_n)-f(\theta_n-\delta_n)}{2\delta_n} -  f'(\theta_n),\\
 \mtermn:=&\frac{\hat f_n(\theta_n+\delta_n)-\hat f_n(\theta_n-\delta_n)}{2\delta_n}\\
 & -\frac{\cost(\theta_n+\delta_n,X_\infty^{\theta_n+\delta_n})-\cost(\theta_n-\delta_n,X_\infty^{\theta_n-\delta_n})}{2\delta_n}.
\end{align*}
The rewriting~\eqref{eq:algoRew2} will facilitate the control of some difference terms.

We consider specific samples $X_\infty^{\theta_n+\delta_n}$ and $X_\infty^{\theta_n-\delta_n}$ from the stationary distribution of the Markov chains with parameters $\theta_n+\delta_n$ and $\theta_n-\delta_n$ respectively, so that for $i=0,\ldots,K-1$, the couplings $\prt{X_{T_n+i\tsimul_n,T_n+(i+1)\tsimul_n}^{\theta_n+\delta_n},X_\infty^{\theta_n+\delta_n}}$ and $\prt{X_{T_n+(K+i)\tsimul_n,T_n+(K+i+1)\tsimul_n}^{\theta_n-\delta_n},X_\infty^{\theta_n-\delta_n}}$ are optimal, as defined in \cite[Remark 4.8]{Peres08}. More precisely, for $i=0$ for example, this means that $X_{T_n,T_n+\tsimul_n}^{\theta_n+\delta_n}$ and $X_\infty^{\theta_n+\delta_n}$ are different with probability $\norm{P_{\theta_n+\delta_n}^{\tsimul_n}(x,\cdot)-m_{\theta_n+\delta_n}}_1$, with $x$ denoting the starting state for the Markov chain simulation between time-steps $T_n$ and $T_n+\tsimul_n$; in the remainder, $\|\cdot\|_1$ denotes the $L_1$ norm.

\subsubsection{Proof of Theorem~\ref{th:as}}

Before delving into the proof, let us first show a preliminary lemma.
\begin{lemma}
\label{lem:variance}
Asumming \ref{as:reg}, the variance conditioned on $\theta$ is finite:
$$\sup_\theta{\rm Var}_\theta(\cost(\theta,X_\infty^\theta))<\infty.$$
\end{lemma}
\begin{proof}
Using Assumption \ref{as:reg}.c, we can write:
\begin{align*}
 {\rm Var}_\theta(\cost(\theta,X_\infty^\theta)) &= \E\brac{ \prt{\cost(\theta,X_\infty^\theta) - \E\brac{\cost(\theta,X_\infty^\theta)}}^2 \mid \theta} \\
 &\leq 
 \E\brac{ \prt{\costOg(\theta,X_\infty^\theta) - \E\brac{\costOg(\theta,X_\infty^\theta)}}^2 \mid \theta}\\
 &\leq 4 \costOg_{\max}^2.
\end{align*}
\end{proof}

The proof of Theorem~\ref{th2} is divided in the following steps:
\begin{enumerate}
 \item We introduce the continuous-time interpolation $\bar \theta$ of the discrete process $\theta$.
 \item We show that $\bar \theta$ is an APT (asymptotic pseudotrajectory, see below) for the flow induced by $f'$, meaning that it remains ``close'' to a trajectory with flow $f'$.
 \item We then show that $\bar \theta$ is a precompact APT, i.e., $f(\bar \theta_t)$ remains bounded.
 \item We deduce that $\bar \theta$ and therefore $\theta$ share the same equilibrium set as the trajectories induced by the flow $f'$, which we assumed to be a single point: the minimum of $f$.
\end{enumerate}

\begin{proof}
We first remind the definition of an asymptotic pseudotrajectory (APT) in our setup.

\begin{definition}[Asymptotic pseudotrajectory, \cite{Benaim99}]
Let the continuous map 
\begin{align*}
\label{eq:flow}
\Phi &: \R_+ \times \R \to \R \\
(t,\theta) &\mapsto \Phi(t,\theta)=\Phi_t(\theta)
\end{align*}
be a semiflow, so that $\Phi_0$ is the identity and $\Phi_{t+s}=\Phi_t \circ \Phi_s$.

A continuous function $X$ is an asymptotic pseudotrajectory for the semiflow $\Phi$ if for any $T>0$:
\begin{equation*}
\label{eq:APT}
 \lim_{t\to\infty} \sup_{0\leq h \leq T} \abs{X(t+h)-\Phi_h(X(t))}=0.
\end{equation*}
Moreover, in $\R$, the asymptotic pseudotrajectory $X$ is said to be precompact if its image is bounded.
\end{definition}

Rather than dealing with the discrete trajectory~$(\theta_n)_{n \in \N}$, we consider its piecewise linear interpolated counterpart in continuous time, $(\bar \theta_t)_{t \in \R_+}$. This is defined as follows
 \begin{equation}
  \label{eq:interpol}
  \begin{cases}
   t_n=\sum_{i=1}^n \stepsize_i, \\
   \bar \theta_{t_n+s} = \theta_n + s\frac{\theta_{n+1}-\theta_n}{\stepsize_{n+1}} & \text{ for } 0\leq s < \stepsize_{n+1}.
  \end{cases}
 \end{equation}
Then, the goal is to use the following proposition from \cite{Benaim99}.
 \begin{proposition}[Proposition 4.1 in \cite{Benaim99} ]
Assume that $f'$ is Lipschitz, and that with probability $1$, for all $T>0$:
\begin{equation}
  \label{eq:bias_vanish}
  \lim_{n \to \infty} \sup_{k\in E_n}\croc{\abs{\sum_{i=n+1}^{k} \stepsize_i (\btermi+\mtermi+\ntermi)}}=0,
 \end{equation}
 where $E_n:=\croc{k \mid t_n < t_k \leq t_n+T}$ is the set of discrete-time timesteps $k$ such that the $t_k$ are within a timeframe $T$ of $t_n$. Then the interpolated process $(\bar \theta_t)_{t \in \R}$ is an asymptotic pseudotrajectory for the flow induced by $f'$.
\end{proposition}

Let us first show Equation~\eqref{eq:bias_vanish}.
We deal with each of error term in~\eqref{eq:bias_vanish} separately. For any $T>0$ and any $k \in E_n$:
\begin{align*}
 \abs{\sum_{i=n+1}^{k} \stepsize_i \btermi}
 &\leq \sum_{i \in E_n} \stepsize_i \abs{\btermi}\\
 &\leq C_0 \sum_{i \in E_n} \stepsize_i \delta_i^2\\
 &\leq C_0 \delta_n^2 \sum_{i \in E_n} \stepsize_i
 \leq C_0 T \delta_n^2 \underset{n \to \infty}{\longrightarrow} 0
\end{align*}
where in the second inequality we have used the bound \eqref{eq:taylor2}, and in the third and last inequalities, we have used the definition of $E_n$ and that $(\delta_n)_n$ is decreasing from the parametrization properties \eqref{eq:parameter1}-\eqref{eq:parameter2}.

Now, let us consider the second term.
Let us recall that $X_\infty^{\theta_n+\delta_n}$ and $X_\infty^{\theta_n-\delta_n}$ build optimal couplings. Using Assumption~\ref{as:mixing}, these quantities are different with probability at most $C_1\rho^{\tsimul_n}=C_1 n^{-\alpha}$. Choosing $\alpha >1$,
and using a Borel-Cantelli lemma, we have that almost surely, for $n$ large enough: $X_{T_n,T_n+\tsimul_n}^{\theta_n+\delta_n}=X_\infty^{\theta_n+\delta_n}$. Proceeding similarly for $i=0,\ldots,K-1$ and $\theta_n-\delta_n$, and using a union bound, we obtain \eqref{ais9c0asc},
\begin{figure*}
\begin{align}
\nonumber
 \Pb\prt{\mtermn=0}
 &\geq
 \Pb\prt{\bigcap_{i=0}^{K-1}\prt{X_{T_n+i\tsimul_n,T_n+(i+1)\tsimul_n}^{\theta_n+\delta_n}=X_\infty^{\theta_n+\delta_n}\cap X_{T_n+(K+i)\tsimul_n,T_n+(K+i+1)\tsimul_n}^{\theta_n-\delta_n}=X_\infty^{\theta_n-\delta_n}}}\\
\nonumber
 &=
 1-\Pb\prt{\bigcup_{i=0}^{K-1}\prt{X_{T_n+i\tsimul_n,T_n+(i+1)\tsimul_n}^{\theta_n+\delta_n} \neq X_\infty^{\theta_n+\delta_n}\cup X_{T_n+(K+i)\tsimul_n,T_n+(K+i+1)\tsimul_n}^{\theta_n-\delta_n} \neq X_\infty^{\theta_n-\delta_n}}}\\
\nonumber
 &\geq
 1-\sum_{i=0}^{K-1}\prt{\Pb\prt{X_{T_n+i\tsimul_n,T_n+(i+1)\tsimul_n}^{\theta_n+\delta_n} \neq X_\infty^{\theta_n+\delta_n}}+ \Pb\prt{X_{T_n+(K+i)\tsimul_n,T_n+(K+i+1)\tsimul_n}^{\theta_n-\delta_n} \neq X_\infty^{\theta_n-\delta_n}}}\\
 &\geq
\label{ais9c0asc}
 1 \text{ for $n$ large enough,}
\end{align}
\end{figure*}
i.e., $\mtermn=0$ almost surely.
Therefore, with probability $1$,
\begin{align*}
 \abs{\sum_{i=n+1}^{k} \stepsize_i \mtermi}
 \underset{n \to \infty}{\longrightarrow} 0.
\end{align*}

Finally,
let us consider the last term.
First, we need to show that $\delta_n^2\E\brac{\ntermn^2 \mid \theta_n} < \infty$, so that with the tower rule:
 \begin{equation}
  \label{eq:as_eta}
  \sup_n\delta_n^2\E\brac{\ntermn^2} < \infty.
 \end{equation}
Let us calculate
 \begin{multline*}
  \E\brac{\prt{\cost(\theta_n+\delta_n,X_\infty^{\theta_n+\delta_n})-f(\theta_n+\delta_n)}^2 \mid \theta_n}
  \\= {\rm Var}_{\theta_n+\delta_n} \prt{\cost(X_{\infty}^{\theta_n+\delta_n})},
 \end{multline*}
so that, with the same reasoning for $\theta_n-\delta_n$, using that for any $a,b \in \R$, $(a+b)^2 \leq 2a^2+2b^2$, we get
\begin{multline*}
\sup_n 4\delta_n^2 \E\brac{\ntermn^2 \mid \theta_n}
\leq 2\sup_n \delta_n^2 {\rm Var}_{\theta_n+\delta_n} \prt{\cost(X_{\infty}^{\theta_n+\delta_n})}
\\ + 2\sup_n \delta_n^2 {\rm Var}_{\theta_n-\delta_n} \prt{\cost(X_{\infty}^{\theta_n-\delta_n})}
 < \infty,
\end{multline*}
where in the final step we have used Lemma~\ref{lem:variance}.

Using 
the parametrization property \eqref{eq:parameter1} and \cite[Theorem 12.1]{Williams91} on the convergence of martingales in $\mathcal{L}^2$, as $M_n:=\sum_{i=1}^n \stepsize_i \ntermi$ defines a martingale $M$ with $$\sum_i \stepsize_i^2 \E \brac{\ntermi^2} \leq \sup_n \delta_n^2 \E\brac{\ntermn^2 \mid \theta_n} \sum_i \stepsize_i^2 \delta_i^{-2}< \infty,$$ then almost surely, $M_n \to M_{\infty}$.
Then, for any $n\in \N$ and $k\in E_n$, with probability $1$:
\begin{align*}
\abs{\sum_{i=n+1}^k \stepsize_i \ntermi} & \leq
\abs{\sum_{i \geq n} \stepsize_i \ntermi}
+\abs{\sum_{i >\sup E_n} \stepsize_i \ntermi}\\
& \underset{n \to \infty}{\longrightarrow} 0.
\end{align*}
Overall, summing the $\btermi$, $\mtermi$ and $\ntermi$, we have proved that \eqref{eq:bias_vanish} holds with probability $1$. Now, we show that $\sup_n \abs{\theta_n} < \infty$.


For $n$ large enough, with probability $1$, $\abs{\btermn+\ntermn+\mtermn} < r/2,$ with $r$ defined in Assumption~\ref{as:reg}.d, so that  if $|\theta_n|>L$, then equation~\eqref{eq:algoRew2} gives $|\theta_{n+1}| < |\theta_n|-\stepsize_n\frac{r}{2}$. Otherwise, if $|\theta_n| < L$, then $\abs{f'(\theta_n)} \leq \sup_{\theta \in [-L,L]} \abs{f'(\theta)}$ (by Assumption~\ref{as:reg}.b) and with equation ~\eqref{eq:algoRew2} $|\theta_{n+1}| \leq |\theta_n| +\stepsize_n(\sup_{\theta \in [-L,L]} \abs{f'(\theta)}+r/2) \leq L +\stepsize_0(\sup_{\theta \in [-L,L]} \abs{f'(\theta)}+r/2)$, as $(\stepsize_n)$ is decreasing (by the parametrization property~\eqref{eq:parameter2}).

Therefore, with probability $1$, $(\theta_n)_n$ remains bounded and we can apply \cite{Benaim99}[Proposition~4.1], so that $\bar \theta$ is an asymptotic pseudotrajectory of the flow $\Phi$ induced by $f'$. It is precompact as $f$ is continuous and $\croc{\theta_n, n \in \N}$ is bounded with probability $1$.

We remark that the flow $\Phi$ satisfies $\frac{d\Phi_t(\theta)}{dt}= -f'(\Phi_t(\theta))$. Since
$$\frac{d}{dt} \brac{f(\Phi_t(\theta))} = \frac{d\Phi_t(\theta)}{dt} \times f'(\Phi_t(\theta)) = -f'(\Phi_t(\theta))^2 <0, $$ 
then $f$ is a Lyapounov function of the flow $\Phi$. Using \cite[Corollary 6.6]{Benaim99} and the uniqueness Assumption~\ref{as:uniqueness}, we get that $\bar \theta_t \overset{a.s.}{\to} \theta^*$ as it is the only minimum, and thus $\theta_n \overset{a.s.}{\to} \theta^*$ as desired.
\end{proof}

\subsubsection{Proof of Theorem~\ref{th2}}


Our approach consists in decomposing the term of interest in multiple terms and then in bounding each term individually.
In the decomposition, the main innovation and difficulty will come from the non-stationary term.
More specifically, letting $\xi_n:=\E\brac{\prt{\theta_n-\theta^*}^2}$, we write
\begin{align}
 \xi_{n+1} &= \xi_n + 2\E\brac{\prt{\theta_{n+1}-\theta_n} \prt{\theta_n-\theta^*}} + \E\brac{\prt{\theta_{n+1}-\theta_n}^2} \nonumber\\
 &\leq
 \xi_n \nonumber\\
 &\phantom{\leq}
 -2\stepsize_n\E\brac{f'(\theta_n) \prt{\theta_n-\theta^*}} \label{eq:comp_main}\\
  &\phantom{\leq}
 -2\stepsize_n\E\brac{\btermn \prt{\theta_n-\theta^*}} \label{eq:comp_bias}\\
  &\phantom{\leq}
  -2\stepsize_n\E\brac{\ntermn \prt{\theta_n-\theta^*}} \label{eq:comp_noise}\\
  &\phantom{\leq}
  -2\stepsize_n\E\brac{\mtermn \prt{\theta_n-\theta^*}} \label{eq:comp_trans}\\
   &\phantom{\leq}
   + \E\brac{\prt{\theta_{n+1}-\theta_n}^2}. \label{eq:comp_last}
\end{align}
In the remainder, we bound the five terms above following these lines:
\begin{itemize}

 \item
To bound \eqref{eq:comp_main}, we use Assumption \ref{as:convex} as in classical gradient descent algorithms.
With the strong convexity from Assumption~\ref{as:convex}, we obtain
\begin{align}
\nonumber
-2\stepsize_n\E\brac{f'(\theta_n)(\theta_n-\theta^*)}
& \leq - 2 \stepsize_n \kappa \E\brac{(\theta_n-\theta^*)^2}\\
\label{eq:convex_case}
& \leq -2\stepsize_n \kappa \xi_n.
\end{align}

 \item
To bound~\eqref{eq:comp_bias}, using Assumption~\ref{as:reg} and the Cauchy-Schwarz inequality, we obtain
 \begin{align*}
  &-2a_n\E\brac{\btermn \prt{\theta_n-\theta^*}}\\
  &\leq  2 \stepsize_n \E\brac{\btermn^2}^{1/2} \E\brac{\prt{\theta_n-\theta^*}^2}^{1/2} \\
  &\leq 2C_0\stepsize_n \delta_n^2\xi_n^{1/2},
 \end{align*}
where we used that $\btermn \leq C_0 \delta_n^2$ by assumption on $f$ and by its Euler discretization around any parameter. More precisely, we can write the Taylor expansion, for $\eps=1$ or $\eps=-1$, as
\begin{equation}
 \label{eq:taylor}
  f(\theta_n+\eps \delta_n) = f(\theta_n) + f'(\theta_n)\eps \delta_n +  f''(\theta_{\eps}) \frac{\delta_n^2}{2},
\end{equation}
for some $\theta_\eps \in [\theta_n-\delta_n,\theta_n+\delta_n],$ so that:
\begin{equation}
 \label{eq:taylor2}
 \btermn \leq \delta_n \frac{\abs{f''(\theta_{+1})-f''(\theta_{-1})}}{2} \leq \delta_n^2 \sup_{\theta \in \Theta} f^{(3)}(\theta).
\end{equation}

\item
To bound~\eqref{eq:comp_noise}, we notice that $\E\brac{\ntermn \mid \theta_n}=0$.  Taking the expectation,  we get $-2\stepsize_n\times$ $\E\brac{\ntermn \prt{\theta_n-\theta^*}}=0.$

\item
To bound~\eqref{eq:comp_trans}, using the ergodicity structure in Assumption~\ref{as:mixing}, the Assumption~\ref{as:reg} and first with a Cauchy-Schwarz inequality:
\begin{align*}
& -2 \stepsize_n \mathbb E \brac{\mtermn  (\theta_n-\theta^*)}\\
&\leq 2 \stepsize_n \E\brac{\mtermn^2}^{1/2} \E\brac{\prt{\theta_n-\theta^*}^2}^{1/2}\\
&=2 \stepsize_n \E\brac{\prt{\theta_n-\theta^*}^2}^{1/2} \E\brac{\E\brac{\mtermn^2 \mid \theta_n}}^{1/2}.
\end{align*}
With
i) Assumption~\ref{as:reg}.c on the instant cost function,
ii) using that $(\sum_{i=0}^{K-1} a_i)^2 \leq K\sum_{i=0}^{K-1}a_i^2$  for any $a_i\in \R$, which follows by a linear algebra argument, and
iii) letting in our case $a_i~:=~\costOg(\theta_n+\delta_n,X_{T_n+i \tsimul_n,T_n+(i+1)\tsimul_n})-\costOg(\theta_n+\delta_n,X_{\infty}^{\theta_n+\delta_n})$, we obtain
\begin{align*}
& \E\brac{\prt{\hat f_n(\theta_n+\delta_n)-\cost(\theta_n+\delta_n,X_{\infty}^{\theta_n+\delta_n})}^2 \mid \theta_n}\\
 &=
 \frac{1}{K^2}\E\brac{\prt{\sum_{i=0}^{K-1} a_i}^2 \mid \theta_n }
 \leq
 \frac{1}{K}\sum_{i=0}^{K-1}\E\brac{ a_i^2 \mid \theta_n}\\
 &\leq
  \frac{\costOg_{\max}^2}{K}\sum_{i=0}^{K-1} \norm{P_{\theta_n+\delta_n}^{\tsimul_n}(x_{\theta_n+\delta_n}^{(i,+)},\cdot)-m_{\theta_n+\delta_n}}_{1},
\end{align*}
where $x_{\theta_n+\delta_n}^{(i,+)}$ denotes the initial state for the $i-$th simulation of the Markov chain with parameter $\theta_n+\delta_n$ at episode $n$.
Applying the same argument, a similar bound is obtained for the Markov chain with parameter $\theta_n-\delta_n$.

Now, using that $(a+b)^2 \leq 2a^2+2b^2$  for any $a,b \in \R$, for the conditional expectation we obtain
\begin{align*}
& 4\delta_n^2\E\brac{\mtermn^2 \mid \theta_n} \\
&\leq
 2\E\Big[\prt{\hat f_n(\theta_n+\delta_n)-\cost(\theta_n+\delta_n,X_{\infty}^{\theta_n+\delta_n})}^2 \\
 & \quad+ \prt{\hat f_n(\theta_n-\delta_n)-\cost(\theta_n-\delta_n,X_{\infty}^{\theta_n-\delta_n})}^2 \mid \theta_n\Big]\\
&\leq
 \frac{2\costOg_{\max}^2}{K}\sum_{i=0}^{K-1} \norm{P_{\theta_n+\delta_n}^{\tsimul_n}(x_{\theta_n+\delta_n}^{(i,+)},\cdot)-m_{\theta_n+\delta_n}}_{1}\\
&\quad +\norm{P_{\theta_n-\delta_n}^{\tsimul_n}(x_{\theta_n-\delta_n}^{(i,-)},\cdot)-m_{\theta_n-\delta_n}}_{1}\\
&\leq
 2\costOg_{\max}^2 C_1\rho^{\tsimul_n},
\end{align*}
where in the last inequality we have used ergodicity (Assumption~\ref{as:mixing}).
We continue the computations of the bound on \eqref{eq:comp_trans}:
\begin{align*}
-2 \stepsize_n \mathbb E \brac{\mtermn  (\theta_n-\theta^*)}
&\leq 
\sqrt{2} \frac{\stepsize_n}{\delta_n}\xi_n^{1/2} \costOg_{\max} C_1^{1/2}\rho^{\tsimul_n/2}. 
\end{align*}
Choosing $T_n:=2\alpha\frac{\log n}{\log 1/\rho}$, we finally obtain that
\begin{equation*}
 -2 \stepsize_n \mathbb E \brac{\mtermn  (\theta_n-\theta^*)} \leq  \sqrt{2} \frac{\gamma_n}{\delta_n}\xi_n^{1/2} \costOg_{\max} C_1^{1/2} n^{-\alpha}. 
\end{equation*}

\item
To bound~\eqref{eq:comp_last}, we let
\begin{multline*}
 a_i := \costOg(\theta_n+\delta_n,X_{T_n+i \tsimul_n,T_n+(i+1)\tsimul_n}) \\
 -\costOg(\theta_n-\delta_n,X_{T_n+(K+i) \tsimul_n,T_n+(K+i+1)\tsimul_n})
\end{multline*}
to obtain
\begin{align*}
& \E\brac{\prt{\theta_{n+1}-\theta_n}^2} \\
 &=
 \frac{\stepsize_n^2}{4 \delta_n^2} \E\brac{\prt{\hat f_n(\theta_n+\delta_n)-\hat f_n(\theta_n-\delta_n)}^2}\\
 &=
 \frac{\stepsize_n^2}{4 \delta_n^2} \E\brac{\prt{\frac{1}{K}\sum_{i=0}^{K-1}a_i+\pen(\theta_n+\delta_n)-\pen(\theta_n-\delta_n) }^2}\\
 &\leq
 \frac{\stepsize_n^2}{2 \delta_n^2} \E\brac{\prt{\frac{1}{K}\sum_{i=0}^{K-1}a_i}^2+\prt{\pen(\theta_n+\delta_n)-\pen(\theta_n-\delta_n) }^2}\\
 &\leq
 \frac{\stepsize_n^2}{2K \delta_n^2} \E\brac{\sum_{i=0}^{K-1}a_i^2} + \frac{C_2\stepsize_n^2}{2} \E\brac{(\theta_n-\theta^*)^2}\\
 & \leq \frac{2\stepsize_n^2 \costOg_{\max}^2}{\delta_n^2} + \frac{C_2\stepsize_n^2}{2} \xi_n
\end{align*}
 \end{itemize}
where the second inequality follows by Assumption~\ref{as:reg}.c.


Finally, summing the previous terms, we obtain
 $$\xi_{n+1}\leq (1-q_n)\xi_n+ u_n \xi_{n}^{1/2}+v_n,$$ where
 $q_n:=\stepsize_n \prt{2\kappa-\frac{C_2\stepsize_n}{2}}$, $u_n=2C_0 \stepsize_n \delta_n^2+\sqrt{2}\costOg_{\max}C_1^{1/2} \frac{\stepsize_n}{\delta_n}n^{-\alpha}$ and $v_n=2\frac{\costOg_{\max}^2 \stepsize_n^2}{\delta_n^2}$.

We can already choose $\alpha$ such that $\frac{n^{-\alpha}}{\delta_n}\leq \delta_n^2$. In this case, $u_n\leq C_u \delta_n^2 \stepsize_n$ with $C_u:=2C_0+\sqrt{2}\costOg_{\max} C_1^{1/2}$.
Now, let us consider the following lemma from~\cite{Walton22}, which we state here for convenience.
{This lemma is purely algebraic and will let us identify the ``correct'' scaling for the input parameters of Algorithm~\ref{algo:kw}.}

\begin{lemma}[\cite{Walton22}]
\label{lem:limsup}
Let $\xi_n$ be a positive sequence, $A_n$ be a sequence, and $B_n,C_n$ be positive non-increasing sequences such that
\begin{equation*}
  \xi_{n+1}\leq \xi_n (1-A_n) + A_n B_n \xi_{n}^{1/2}+C_n A_n.
\end{equation*}
Then,
\begin{equation*}
 \limsup_{n \to \infty} \frac{\xi_n}{B_n^2+C_n} \leq \frac{1}{2}.
\end{equation*}
\end{lemma}

Since $\stepsize_n < \frac{4 \kappa}{C_2}$, then $q_n$ is positive and Lemma~\ref{lem:limsup} gives the scaling:
 \begin{equation*}
  \limsup_{n \to \infty}  \frac{\xi_n}{B_n^2+C_n} \leq \frac{1}{2},
 \end{equation*}
 with $B_n=\frac{2C_u \delta_n^2}{4\kappa-C_2\stepsize_n}$ and $C_n=\frac{4\costOg_{\max}^2\stepsize_n}{\delta_n^2 \prt{4\kappa-C_2\stepsize_n}}$.
Therefore, to get comparable scalings, we choose the parameters such that $ \delta_n^4$ and $\stepsize_n \delta_n^{-2}$ are of the same order. A valid parameter choice is obtained when $\stepsize_n = n^{-1}$, $\delta_n = n^{-1/6}$ and $\alpha > 1/2$.
Within these parameters, we get
 \begin{equation}
 \label{eq:scale_conv}
 \limsup_{n \to \infty}\xi_n n^{2/3} \leq \frac{\prt{2C_0+\sqrt{2}\costOg_{\max}C_1^{1/2}}^2}{8\kappa^2}+\frac{\costOg_{\max}^2}{2\kappa}.
 \end{equation}
This concludes the proof.

\bibliographystyle{plain}

\begin{IEEEbiographynophoto}{Jonatha Anselmi}
is a tenured researcher at the French National Institute for Research in Digital Science and Technology (Inria), since 2014. Prior
to this, he was a full-time researcher at the Basque Center for Applied
Mathematics and a postdoctoral researcher at Inria. He received his PhD
in computer engineering at Politecnico di Milano (Italy) in 2009.
At the intersection of applied mathematics, computer science and engineering, his research interests focus on decision making under uncertainty, with particular emphasis on the development of highly-scalable algorithms that minimize congestion and operational costs of large-scale distributed systems.
\end{IEEEbiographynophoto}

\begin{IEEEbiographynophoto}{Bruno Gaujal}
is an Inria researcher.  Till Dec. 2015,  he has been  the head of the large-scale computing team in Inria Grenoble-Alpes. He has held several positions in AT\&T Bell Labs, Loria and École Normale Supérieure of Lyon. He obtained his PhD from University of Nice in 1994. He is a founding partner of a start-up company, RTaW, since 2007.  His main interests are in performance evaluation, optimization and control of large discrete event dynamic systems with applications to telecommunication and large computing infrastructures.
\end{IEEEbiographynophoto}

\begin{IEEEbiographynophoto}{Louis-Sebastien Rebuffi}
defended his PhD thesis in 2023 at Université Grenoble Alpes
under the supervision of Bruno Gaujal and Jonatha Anselmi.
His research interests focus on reinforcement learning algorithms applied to controlled queuing systems, viewing them as Markov decision processes.
\end{IEEEbiographynophoto}


\vfill

\end{document}